\documentclass[reqno,oneside,11pt]{amsart}

\usepackage[a4paper, hmargin={2.8cm, 2.8cm}, vmargin={2.5cm, 2.5cm}]{geometry}  
\usepackage[danish,english]{babel} 
\usepackage[latin1]{inputenc} 
\usepackage[T1]{fontenc} 
\usepackage{lmodern}
\usepackage{amsmath,amsthm,amssymb,amsfonts,mathrsfs,latexsym} 
\usepackage{graphicx} 
\usepackage{verbatim} 
\usepackage[all]{xy} 
\usepackage[pagebackref, colorlinks, linkcolor=red, citecolor=blue, urlcolor=blue, hypertexnames=true]{hyperref}  
\usepackage{multirow}
\usepackage{color}
\usepackage{enumerate}

\makeatletter
\newtheorem*{rep@theorem}{\rep@title}
\newcommand{\newreptheorem}[2]{%
\newenvironment{rep#1}[1]{%
 \def\rep@title{#2 \ref{##1}}%
 \begin{rep@theorem}}%
 {\end{rep@theorem}}}
\makeatother

\newtheorem{thm}{Theorem}[section]

\newreptheorem{thm}{Theorem}
\newtheorem{lem}[thm]{Lemma}
\newtheorem{prop}[thm]{Proposition}
\newtheorem{cor}[thm]{Corollary}
\newtheorem*{thm*}{Theorem}
\newtheorem*{problem*}{Problem}
\newtheorem{question}[thm]{Question}

\newtheorem*{claim*}{Claim}
\theoremstyle{definition}
\newtheorem{defi}[thm]{Definition}
\newtheorem{exem}[thm]{Example}

\newtheorem{rem}[thm]{Remark}

\newtheorem{conjecture2}[thm]{Conjecture}
\newcommand{\claimmark}{\hfill$\lozenge$}

\newcommand{\C}{\mathbb{C}}
\newcommand{\N}{\mathbb{N}}

\newcommand{\R}{\mathbb{R}}
\newcommand{\Z}{\mathbb{Z}}
\newcommand{\T}{\mathbb{T}}

\newcommand{\ra}{\rightarrow}
\renewcommand{\epsilon}{\varepsilon}
\renewcommand{\phi}{\varphi}

\renewcommand{\bar}{\overline}

\newcommand{\Manoa}{M\=anoa}

\newcommand{\Hawaii}{Hawai\kern.05em`\kern.05em\relax i}

\DeclareFontFamily{U}{mathx}{\hyphenchar\font45}
\DeclareFontShape{U}{mathx}{m}{n}{
      <5> <6> <7> <8> <9> <10>
      <10.95> <12> <14.4> <17.28> <20.74> <24.88>
      mathx10
      }{}
\DeclareSymbolFont{mathx}{U}{mathx}{m}{n}
\DeclareFontSubstitution{U}{mathx}{m}{n}
\DeclareMathAccent{\widecheck}{0}{mathx}{"71}
\DeclareMathAccent{\wideparen}{0}{mathx}{"75}

\numberwithin{equation}{section}

\begin{document}
\selectlanguage{english} 


\title[Low dimensional properties of uniform Roe algebras]{\texorpdfstring{Low dimensional properties of uniform Roe algebras}{Low dimensional properties of uniform Roe algebras}}


\author{Kang Li$^{1}$}
\address{Mathematisches Institut der WWU M\"unster,
\newline Einsteinstrasse 62, 48149 M\"unster, Germany}
\email{lik@uni-muenster.de}

\author{Rufus Willett$^{2}$}
\address{Department of Mathematics,
University of \Hawaii~at \Manoa,
\newline Honolulu, HI, 96822, U.S.A.}
\email{rufus@math.hawaii.edu}

\thanks{{$^{1}$} Supported by the Danish Council for Independent Research (DFF-5051-00037) and partially supported by the DFG (SFB 878).}

\thanks{{$^{2}$} Partially supported by the US NSF
(DMS-1401126 and DMS-1564281).}

\begin{abstract}

The goal of this paper is to study when uniform Roe algebras have certain $C^*$-algebraic properties in terms of the underlying space: in particular, we study properties like having stable rank one or real rank zero that are thought of as low dimensional, and connect these to low dimensionality of the underlying space in the sense of the asymptotic dimension of Gromov.  Some of these results (for example, on stable rank one, cancellation, strong quasidiagonality, and finite decomposition rank) give definitive characterizations, while others (on real rank zero) are only partial and leave a lot open.  

We also establish results about $K$-theory, showing that all $K_0$-classes come from the inclusion of the canonical Cartan in low dimensional cases, but not in general; in particular, our $K$-theoretic results answer a question of Elliott and Sierakowski about vanishing of $K_0$ groups for uniform Roe algebras of non-amenable groups.   Along the way, we extend some results about paradoxicality, proper infiniteness of projections in uniform Roe algebras, and supramenability from groups to general metric spaces.  These are ingredients needed for our $K$-theoretic computations, but we also use them to give new characterizations of supramenability for metric spaces.  
\end{abstract}

\date{\today}
\maketitle
\parskip 4pt

\section{Introduction}

Uniform Roe algebras are $C^*$-algebras associated to discrete metric spaces: they reflect the `large-scale' or `coarse' geometry of the space.  This paper studies $C^*$-algebraic properties of uniform Roe algebras: we focus on properties that can be thought of as `low dimensional', as they typically only hold for $C(X)$ when $X$ has `small' covering dimension.  Examples include having stable rank one, having real rank zero, and having `topologically uninteresting' $K$-theory.  Taken together, our results suggest that the $C^*$-algebraic structure of a uniform Roe algebra $C^*_u(X)$ can only be expected to be `low dimensional' when $X$ is low dimensional, where the dimension of $X$ is to be understood in the sense of Gromov's asymptotic dimension.

We work with the following class of metric spaces.

\begin{defi}
A metric space $X$ has \emph{bounded geometry} if for any $r>0$ there is a uniform bound on the cardinalities of all $r$-balls in $X$.
\end{defi}

Countable discrete groups give rise to an important class of bounded geometry metric spaces: any such group admits a bounded geometry metric for which the left translation action of the group on itself is by isometries; moreover, any two such metrics give rise to the same large-scale geometric structure on $X$ (see for example \cite[Proposition 2.3.3]{Willett:2009rt}).  Essentially everything we study in this paper will be insensitive to the differences between such choices of metric, so we may canonically view groups as metric objects in this way.

Here is the definition of asymptotic dimension.

\begin{defi}\label{asdim}
A collection $\mathcal{U}$ of subsets of a metric space is \emph{$r$-separated} if for all distinct $A,B$ in $\mathcal{U}$ we have that $d(A,B)>r$.  A bounded geometry metric space has \emph{asymptotic dimension at most $d$} for some $d\in \N_0$ if for any $r>0$ there exists a decomposition 
$$
X=U_0\sqcup \cdots \sqcup U_d
$$
where each $U_i$ in turn decomposes into a uniformly bounded collection of $r$-separated subsets.
\end{defi}

We will mainly be interested in the special cases when $X$ has asymptotic dimension zero or one in this paper.

\begin{defi}
Let $X$ be a bounded geometry metric space, and let $\C_u[X]$ denote the $*$-algebra of all $X$-by-$X$ matrices $(a_{xy})$ with uniformly bounded entries in $\C$, and such that the \emph{propagation} 
$$
\text{prop}(a):=\sup\{d(x,y)\mid a_{xy}\neq 0\}
$$
is finite.  As $X$ has bounded geometry, $\C_u[X]$ acts on $\ell^2(X)$ by bounded operators.  The \emph{uniform Roe algebra} $C^*_u(X)$ of $X$ is the operator-norm closure of the image of $\C_u[X]$ in this representation.
\end{defi}

When $X=G$ is a countable discrete group equipped with a left-invariant bounded geometry metric we write $C^*_u|G|$ for the uniform Roe algebra of $G$ to avoid any possible confusion with the group $C^*$-algebras of $G$; this standard notation is meant to suggest that $|G|$ is the metric object underlying the group $G$.  In this case, there is a well-known isomorphism $C^*_u|G|\cong \ell^\infty(G)\rtimes_r G$, where the action of $G$ on $\ell^\infty(G)$ is induced by the right translation action of $G$ on itself: see for example \cite[Proposition 5.1.3]{Brown:2008qy}.

Our first goal in this paper is to show that many natural $C^*$-algebraic properties of $C^*_u(X)$ are equivalent to the space $X$ having asymptotic dimension zero.   The following theorem states the results actually proved in this paper; our real goal, however, is the corollary that comes afterwards, which follows on combining the theorem with well-known general results (see Section \ref{asdim0 sec} for background and references).

\begin{thm}
Let $X$ be a bounded geometry metric space. Then if $X$ has asymptotic dimension zero, $C^*_u(X)$ is an inductive limit of finite dimensional $C^*$-algebras.  

Moreover, assume $C^*_u(X)$ has any of the following properties: finite decomposition rank; strong quasidiagonality; cancellation.  Then $X$ has asymptotic dimension zero.
\end{thm}

\begin{cor}
Let $X$ be a bounded geometry metric space.   The following are equivalent:
\begin{enumerate}
\item $X$ has asymptotic dimension zero;
\item $C^*_u(X)$ is an inductive limit of finite dimensional $C^*$-subalgebras;
\item $C^*_u(X)$ is AF in the local sense;
\item $C^*_u(X)$ is strongly quasidiagonal;
\item $C^*_u(X)$ has decomposition rank zero;
\item $C^*_u(X)$ has finite decomposition rank;
\item $C^*_u(X)$ has nuclear dimension zero;
\item $C^*_u(X)$ has stable rank one;
\item $C^*_u(X)$ has cancellation.
\end{enumerate}
\end{cor}

This is related, but complementary, to results of Scarparo \cite{Scarparo:2016kl} and Wei \cite{Wei:2011kl}, who study finiteness properties for uniform Roe algebras.

Spaces with asymptotic dimension zero are rather rare: for example, a countable group has asymptotic dimension zero if and only it is locally finite (this means that any finite subset generates a finite subgroup).  Thus the corollary says that the properties involved only hold for uniform Roe algebras of a very restricted class of spaces.

We also study real rank zero.  Here we are not able to say anything definitive, but at least get the following result in Section \ref{rr0 sec} by using $K$-theoretic methods to exhibit projections in quotient algebras that do not lift (the second author thanks Chris Phillips for suggesting studying liftings of projections in relation to this problem).

\begin{thm}
The uniform Roe algebra of $\Z^2$ does not have real rank zero.  
\end{thm}

This can be bootstrapped up to show that for many natural spaces $X$, $C^*_u(X)$ does not have real rank zero.  Nonetheless, we leave open the interesting questions of whether $C^*_u|\Z|$ or $C^*_u(\N)$ have real rank zero.  

Our next goals are $K$-theoretic.   First, note that the uniform Roe algebra of a bounded geometry space $X$ contains a copy of $\ell^\infty(X)$ as diagonal matrices.  The following theorem from Section \ref{asdim1 sec} extends work of Elliott and Sierakowski \cite[Section 5]{MR3508496} on the case when $X$ is a free group.  

\begin{thm}
Let $X$ be a metric space with bounded geometry and let $\delta:{\ell}^\infty(X)\ra C_u^*(X)$ be the inclusion of the diagonal. If the asymptotic dimension of $X$ is at most one, then the induced map
\begin{align*}
\delta_*: K_0({\ell}^\infty(X))\ra K_0(C_u^*(X))
\end{align*}
is surjective.
\end{thm}

Spaces with asymptotic dimension one are again quite special: examples include $\Z$, non-abelian free groups, and various other `tree-like' spaces.  For such spaces, the theorem above says that the $K_0$ groups of $C^*_u(X)$ do not contain any interesting topological information: indeed, from the point of view of the (uniform) coarse Baum-Connes conjecture (see \cite{Spakula:2009tg} or \cite{Engel:2014tx}),  the image of $\delta_*$ is precisely the zero-dimensional part of $K_0(C^*_u(X))$.

Using some results about properly infinite projections in uniform Roe algebras from Section \ref{proj sec} (where we also extend some results about supramenability and properly infinite projections from groups to metric spaces), this has the following corollary, which again generalizes a computation of Elliott and Sierakowski for finitely generated non-abelian free groups.

\begin{cor}
Let $X$ be a non-amenable bounded geometry metric space with asymptotic dimension at most one.  Then $K_0(C^*_u(X))=0$.
\end{cor}

Examples include finitely generated non-abelian free groups, many other more `exotic' groups such as the lamplighter-type group $(\Z/2\Z)\wr F_2$, and also any bounded degree tree with all vertices of degree at least $3$.

Finally, in Section \ref{high dim sec}, we show that the analogues of the theorem and corollary above fail for many spaces of higher dimension; in other words, the $K_0$ groups of such uniform Roe algebras do see interesting topological structure.  Without stating our results in maximal generality in this introduction, here is a sample theorem.

\begin{thm}
Let $X$ be the fundamental group of a closed orientable surface (for example, $X=\Z^2$, the fundamental group of the two-torus).  Then the map
\begin{align*}
\delta_*: K_0({\ell}^\infty(X))\ra K_0(C_u^*(X))
\end{align*}
is not surjective.
\end{thm}

In particular, when $X$ is the fundamental group of a surface of genus $2$ this answers a question of Elliott and Sierakowski: they asked in \cite[Question 5.2]{MR3508496} whether $K_0(C^*_u|G|)$ is always zero for non-amenable groups, and the above shows the answer is `no'.  The proof of the above theorem uses higher index theory to construct interesting classes in $K_0(C^*_u(X))$.

\section{Consequences of asymptotic dimension zero}\label{asdim0 sec}

In this section, we show that some standard $C^*$-algebraic properties of $C^*_u(X)$ are equivalent to $X$ having asymptotic dimension zero.  First, here is a definition of asymptotic dimension zero that is particularly well-suited to our purposes; we leave checking that it is equivalent to Definition \ref{asdim} from the introduction as an exercise for the reader.

\begin{defi}\label{ad0def}
Let $X$ be a bounded geometry metric space.  For each $r\geq 0$, let $\sim_r$ be the equivalence relation generated by the relation: $xRy$ if $d(x,y)\leq r$.  The space $X$ has \emph{asymptotic dimension zero} if for every $r\geq 0$, the relation $\sim_r$ has uniformly finite equivalence classes.  In symbols, if $[x]_r$ is the equivalence class of $x\in X$ for $\sim_r$, then $X$ has asymptotic dimension zero if $\sup_{x\in X}|[x]_r|<\infty$.
\end{defi}

For example, it is not too difficult to check that if $X=G$ is a countable group equipped with a left invariant bounded geometry metric, then $X$ has asymptotic dimension zero if and only if $G$ is locally finite: see \cite[Theorem~2]{MR2237596}.

Here is the main theorem of this section.  Precise definitions of all but two of the $C^*$-algebraic concepts involved can be found in \cite{MR2188261}; the exceptions are decomposition rank and nuclear dimension, which can be found in \cite{Kirchberg:2004uq}\footnote{The definition of decomposition rank is only stated for separable $C^*$-algebras in \cite{Kirchberg:2004uq}, but makes good sense, and seems interesting, in the non-separable case too.} and \cite{Winter:2010eb} respectively.

\begin{thm}\label{ad0the}
Let $X$ be a bounded geometry metric space.   The following are equivalent:
\begin{enumerate}
\item $X$ has asymptotic dimension zero;
\item $C^*_u(X)$ is an inductive limit of finite dimensional $C^*$-subalgebras;
\item $C^*_u(X)$ is AF in the local sense;
\item $C^*_u(X)$ is strongly quasidiagonal;
\item $C^*_u(X)$ has decomposition rank zero;
\item $C^*_u(X)$ has finite decomposition rank;
\item $C^*_u(X)$ has nuclear dimension zero;
\item $C^*_u(X)$ has stable rank one;
\item $C^*_u(X)$ has cancellation.
\end{enumerate}
\end{thm}

This should be compared to the results of Shuyun Wei \cite{Wei:2011kl}, who shows that quasidiagonality, stable finiteness, and finiteness of $C^*_u(X)$ are all equivalent to the relations $\sim_r$ from Definition \ref{ad0def} all having finite\footnote{Not necessarily uniformly finite!} equivalence classes; and also of Eduardo Scarparo \cite{Scarparo:2016kl}, who shows that if $X=G$ is a group, then local finiteness of $G$ is equivalent to finiteness of $C^*_u(X)$ (the proofs of Wei's and Scarparo's results are similar; however, the two were arrived at independently). We also refer to \cite[Corollary~5.4]{Li-Liao} for a summary of equivalent $C^*$-properties of uniform Roe algebras coming from (not necessarily countable) locally finite groups. Moreover, in \cite{Li-Liao} Liao and the first named author also obtained a classification result of uniform Roe algebras of countable locally finite groups.

Proceeding to the proof of Theorem \ref{ad0the}, the implications
\begin{equation}\label{implications}
\xymatrix{ (2) \ar@{=>}[r] & (3) \ar@{=>}[d] \ar@{<=>}[r] & (5) \ar@{=>}[d] \ar@{<=>}[r] & (7) \ar@{=>}[r] & (8) \ar@{=>}[r] & (9) \\  & (4) & (6)  & && & &}
\end{equation}
are well-known to hold for general (unital) $C^*$-algebras.  Indeed: $(2)\Rightarrow(3)$ is trivial; $(3)\Rightarrow (4)$ is straightforward; see \cite[Remarks~2.2 (iii)]{Winter:2010eb} for $(3)\Leftrightarrow(5)\Leftrightarrow (7) $; $(5)\Rightarrow(6)$ is trivial; $(3)\Rightarrow (8)$ is straightforward; and see for example \cite[Proposition~V.3.1.24]{MR2188261} for $(8)\Rightarrow(9)$.   However, none of the one-directional arrows in line \eqref{implications} above are reversible in general.  The implication $(6)\Rightarrow (8)$ is known to hold in some cases, for example for simple, separable, unital $C^*$-algebras by combining the main result of \cite{Winter:2010qj}, \cite[Proposition~5.1]{Kirchberg:2004uq} (to get finiteness), and \cite[Theorem~6.7]{Rordam:2004lw}; it is, however, not even true that $(6)\Rightarrow (9)$ in general, as one can see from the commutative case using \cite[Proposition~3.3]{Kirchberg:2004uq}.  The implication $(6)\Rightarrow (4)$ is known in the separable case \cite[Proposition~5.3]{Kirchberg:2004uq}, but it is not clear to us whether it holds for non-separable $C^*$-algebras in general.  The most subtle point is perhaps that (2) and (3) are equivalent for separable $C^*$-algebras by a result of Bratelli, but not in general as shown by Farah and Katsura in \cite{Farah:2010db}. If we assume the continuum hypothesis then $(3)\Leftrightarrow(2)$ would follow in our case from the fact that $X$ is countable (as it has bounded geometry) and \cite[Theorem~1.5]{Farah:2010db}; the result of Theorem \ref{ad0the} is independent of any assumptions on the continuum hypothesis, however.

To prove Theorem \ref{ad0the} it suffices therefore to show (1) implies (2), (4) implies (1), (6) implies (1), and (9) implies (1).  

\begin{proof}[Proof of Theorem \ref{ad0the}, (1) implies (2)]
Assume that $X$ has asymptotic dimension zero, so for each $r>0$, the equivalence classes $I_r$ for the relation $\sim_r$ from Definition \ref{ad0def} are uniformly bounded (and in particular, uniformly finite by bounded geometry).  Fix a total order on $X$, and for each finite subset $A$ of $X$, let $f_A:A\to \{1,...,|A|\}$ be the order-isomorphism determined by the total order.  

Consider now the collection $\mathcal{I}$ of ordered pairs $(r,\mathcal{P})$, where $r>0$, and $\mathcal{P}=\{P_1,...,P_N\}$ is a partition of $I_r$ into finitely many non-empty sets (which we think of as `colours') such that if $A,B\in I_r$ have the same colour, then $|A|=|B|$.  Fix $(r,\mathcal{P})$ in $\mathcal{I}$ for now.  Let $n_1,...,n_N$ denote the cardinalities of the sets in each of the colours $P_1,...,P_N$ of $\mathcal{P}$, and let $B=\bigoplus_{i=1}^N M_{n_i}(\C)$.   For each $A\in I_r$ with colour $P_i$, let $u_{A,i}:\ell^2(\{1,...,|A|\})\to \ell^2(A)$ be the unitary determined by the order-isomorphism $f_A$.  Define 
$$
\phi:B\to \prod_{A\in I_r}\mathcal{B}(\ell^2(A))\subseteq \mathcal{B}(\ell^2(X)), \quad (a_i)_{i=1}^N \mapsto \prod_{i=1}^N\prod_{A\in P_i} u_{A,i}a_iu_{A,i}^*.
$$
Note that as the sets $A\in I_r$ are uniformly bounded, the image of $\phi$ is contained in $C^*_u(X)$.  We denote by $A_{r,\mathcal{P}}$ the corresponding finite dimensional sub-$C^*$-algebra of $C^*_u(X)$.  Informally, $A_{r,\mathcal{P}}$ consists of operators that are block diagonal with respect to $I_r$, and that are constant on each colour in $\mathcal{P}$.

We now define a partial order on the set $\mathcal{I}$ by $(r,\mathcal{P})\leq (s,\mathcal{Q})$ if $r\leq s$ and $A_{r,\mathcal{P}}\subseteq A_{s,\mathcal{Q}}$.   This is indeed a partial order: it is clearly reflexive and transitive, while antisymmetry follows from the fact that $(r,\mathcal{P})\leq (r,\mathcal{Q})$ if and only if $\mathcal{Q}$ refines $\mathcal{P}$.
We claim that this partial order makes $\mathcal{I}$ into a directed set.  Using the fact just noticed about refinements, it suffices to show that for any $(r,\mathcal{P})\in \mathcal{I}$ and any $s\geq r$, there exists a partition $\mathcal{Q}$ of $I_s$ such that $A_{r,\mathcal{P}}\subseteq A_{s,\mathcal{Q}}$: indeed, given this, we have that if $A_{r_1,\mathcal{P}_1}$ and $A_{r_2,\mathcal{P}_2}$ are arbitrary then for $s=\max\{r_1,r_2\}$ there are partitions $\mathcal{Q}_1$, $\mathcal{Q}_2$ of $I_s$ such that $A_{r_1,\mathcal{P}_1}\subseteq A_{s,\mathcal{Q}_1}$ and $A_{r_2,\mathcal{P}_2}\subseteq A_{s,\mathcal{Q}_2}$; then both $(r_1,\mathcal{P}_1)$ and $(r_2,\mathcal{P}_2)$ are dominated by $(s,\mathcal{Q})$ for our partial order, where $\mathcal{Q}$ is the coarsest common refinement of $\mathcal{Q}_1$ and $\mathcal{Q}_2$.  

To establish the claimed fact, then, say we are given $(r,\mathcal{P})\in \mathcal{I}$ and $s\geq r$.  Write $\mathcal{P}=\{P_1,...,P_m\}$. and let $f_A:A\to \{1,...,|A|\}$ be the bijection above.  Define a colouring $\mathcal{P}_A=\{P_{1,A},...,P_{n,A}\}$ on $\{1,...,|A|\}$ by stating that $n$ is in $P_{i,A}$ if and only if it is in the image of $f_A(C)$ for some $C\subseteq A$ with $C\in P_{i}$ (some sets $P_{i,A}$ could be empty).   To define $\mathcal{Q}$, it suffices to say when $A,B\in I_s$ have the same colour with respect to $\mathcal{Q}$.   Say then that $A,B\in I_s$ have the same colour with respect to $\mathcal{Q}$ if and only if $|A|=|B|$, and if for each $i\in \{1,...,m\}$ we have equality $P_{i,A}=P_{i,B}$ as subsets of $\{1,...,|A|\}=\{1,...,|B|\}$; as there are only finitely many cardinalities of sets in $I_s$, and as there are only finitely many possible partitions of the corresponding sets $\{1,...,n\}$ into at most $m$ colours, we see that $\mathcal{Q}$ is a finite partition.  It is not too difficult to check that $(r,\mathcal{P})\leq (s,\mathcal{Q})$ as required.

To complete the proof, it will suffice to show that the union $\bigcup_{(r,\mathcal{P})\in \mathcal{I}}A_{r,\mathcal{P}}$ is dense in $C^*_u(X)$, and for this it suffices to show that any finite propagation operator can be approximated by an element of this union.  For this, let $\epsilon>0$, and let $a\in C^*_u(X)$ have propagation at most $r$.  Then $a$ is contained in the $C^*$-algebra $\prod_{A\in I_r}\mathcal{B}(\ell^2(A))$, which we identify with $\prod_{A\in I_r}M_{|A|}(\C)$ using the bijections $f_A$.   Write $a_A$ for the component of $a$ in the relevant copy of $M_{|A|}(\C)$.  Set $N=\max\{|A|\mid A\in I_r\}$, and for each $n\in \{1,...,N\}$, choose an $(\epsilon/2)$-dense subset $\{b_{n,1},...,b_{n,M_n}\}$ of the ball of radius $\|a\|$ in $M_n(\C)$; note that for each $A$ there exists $m(A)\in \{1,...,M_n\}$ such that $\|a_A-b_{|A|,m(A)}\|<\epsilon/2$.  For each $n\in \{1,...,N\}$ and $m\in \{1,...,M_n\}$, set 
$$
P_{n,m}=\{A\in I_r\mid |A|=n \text{ and } m(A)=m\},
$$
and define 
$$
\mathcal{P}=\{P_{n,m}\mid n\in \{1,...,N\},~m\in \{1,...,M_n\}, ~P_{n,m}\neq\varnothing\}.
$$
Then we have that the element with entries $b:=b_{|A|,m(A)}$ as $A$ ranges over $I_r$ is in $A_{r,\mathcal{P}}$ and 
$$
\|a-b\|=\sup_{A\in I_r}\|a_A-b_{|A|,m(A)}\|<\epsilon;
$$
we are done.
\end{proof}

Before we start on the proofs that (4), (6), and (9) imply (1), let us describe how the basic idea works under a stronger hypothesis that is maybe a bit more intuitive; this also allows us to introduce some terminology that we will need later.   Assume then that $X$ contains a coarsely embedded copy of $\N$ in the sense of the following definition.
\begin{defi}\label{cemb}
Let $X,Y$ be metric spaces.  A map $f:X\to Y$ is a \emph{coarse embedding} if there are non-decreasing proper functions $\rho_-,\rho_+:[0,\infty)\to[0,\infty)$ such that
$$
\rho_-(d_X(x_1,x_2))\leq d_Y(f(x_1),f(x_2))\leq \rho_+(d_X(x_1,x_2))
$$
for all $x_1,x_2\in X$.  The map $f$ is a \emph{coarse equivalence} if it is a coarse embedding, and if there is a constant $c>0$ such that each $y\in Y$ is within $c$ of some point in the image of $f$.
\end{defi}
\noindent The existence of a coarse embedding of $\N$ into $X$ forces $X$ to have positive asymptotic dimension; moreover such a coarse embedding does exist, for example, whenever $X$ is a countable discrete group of positive asymptotic dimension, or a net in a complete connected, open Riemannian manifold.  Now, in this case, $C^*_u(X)$ will contain a proper isometry: roughly, take the operator that acts as the unilateral shift along the coarsely embedded copy of $\N$, and as the identity everywhere else.  This proves all the remaining implications in this case.  Similar arguments to this also underlie the results of Scarparo \cite{Scarparo:2016kl} and Wei \cite{Wei:2011kl} mentioned above.  

Unfortunately, a space with positive asymptotic dimension need not contain a coarsely embedded copy of $\N$ in general.  The following lemma, which says roughly that $X$ must at least contain longer and longer `coarse line segments', provides a useful substitute. 

\begin{lem}\label{asdim0}
Say $X$ is a bounded geometry metric space which does not have asymptotic dimension zero.  Then there is a number $r>0$ and for each $n\geq 1$ a subset $S_n=\{x_1^{(n)},...,x^{(n)}_{m_n}\}$ of $X$ with the following properties:
\begin{enumerate}
\item $m_n\to\infty$ as $n\to\infty$;
\item for each $n$, and each $i\in \{1,...,m_n-1\}$, $d(x^{(n)}_i,x_{i+1}^{(n)})\leq 2r$, and $d(x^{(n)}_1,x^{(n)}_{i+1})\in [ir,(i+1)r)$;
\item the sequence $(\inf_{m\neq n}d(S_n,S_m))_{n=1}^\infty$ is strictly positive, and tends to infinity as $n$ tends to infinity.
\end{enumerate}
\end{lem}

\begin{proof}
Negating Definition \ref{ad0def} gives the existence of some $r>0$ such that the equivalence relation $\sim_r$ as defined there has equivalence classes of either infinite, or finite and arbitrarily large, cardinalities.  Choose $(S_n)$ recursively as follows.  Fix any point $x_1^{(1)}$ in $X$, and set $S_1=\{x_1^{(1)}\}$.  If $S_1,...,S_N$ have been chosen, consider 
$$
Y:=\Bigg\{x\in X~\Big|~ d\Big(x,\bigcup_{n=1}^N S_n\Big)>N\Bigg\}.
$$
As the set removed from $X$ to define $Y$ is finite, it is still true that the equivalence relation $\sim_r$ on $Y$ (now defined with respect to the induced metric on $Y$) contains equivalence classes with either infinite, or finite and arbitrarily large, cardinalities.   In particular, as $Y$ has bounded geometry there exist points $x,y\in Y$ with $x\sim_r y$, and such that 
$$
d(x,y)>\max\{d(x_1^{(n)},x_{m_n}^{(n)})\mid n\in \{1,...,N\}\}.
$$
As $x\sim_r y$ there is a finite sequence $x=x_1^{(N+1)},x_2^{(N+1)},...,x_{m_{N+1}}^{(N+1)}=y$ with $d(x_i^{(N+1)},x_{i+1}^{(N+1)})\leq 2r$, and such that $d(x^{(n)}_1,x^{(n)}_{i+1})\in [ir,(i+1)r)$ for all $i\in \{1,...,m_{N+1}-1\}$.  This gives us our desired set $S_{N+1}$.
\end{proof}

We are now ready to discuss the proofs of the implications (4) implies (1), and (6) implies (1).  

\begin{proof}[Proof of Theorem \ref{ad0the}, (4) implies (1) and (6) implies (1)]
First note that \cite[line~(3.3)]{Kirchberg:2004uq} implies that finite decomposition rank passes to quotients, and the argument of \cite[Proposition~5.1]{Kirchberg:2004uq} shows that a $C^*$-algebra with finite decomposition rank cannot contain a proper isometry (both of these arguments are only stated for the separable case, but the proofs work in general).  On the other hand, strong quasidiagonality for a $C^*$-algebra implies that all of its quotients are quasidiagonal, and so in particular cannot contain a proper isometry.  

Thus to prove both of these implications it will suffice to exhibit a quotient of $C^*_u(X)$ that contains a proper isometry whenever $X$ has positive asymptotic dimension.  We will need some machinery from \cite[Sections 3 and 4]{Spakula:2014aa}; we refer the reader to that paper for notation and terminology.  For compatibility with that paper, note that we may without loss of generality assume that the metric on $X$ is integer-valued by replacing it with $\lceil d\rceil$ (where $\lceil\cdot \rceil$ is the ceiling function); indeed, this does not affect $C^*_u(X)$, or the asymptotic dimension of $X$.

Assume that $X$ does not have asymptotic dimension zero.  Hence there exists $r>0$ and a sequence $(S_n)$ of subsets of $X$ with the properties in Lemma \ref{asdim0}.  Let $Y=\{x_1^{(n)}\mid n\in \N\}$.  This is an infinite subset of $X$, and thus we may fix a non-principal ultrafilter $\omega$ on $X$ such that $\omega(Y)=1$.  For each $k\geq 2$, let $Y_k=\{x_1^{(n)}\mid m_n\geq k\}$; note that this is a cofinite subset of $Y$, and thus $\omega(Y_k)=1$ for all $k$.  Define a partial translation (in the sense of \cite[Definition~3.1]{Spakula:2014aa}) $t_k:Y_k\to X$ by $t_k(x_1^{(n)})=x_k^{(n)}$.  These partial translations are compatible with $\omega$ in the sense of \cite[Definition~3.2]{Spakula:2014aa}, so we may define $\omega_k=t_k(\omega)$.  Note that if $d_\omega$ is as in \cite[Proposition~3.7]{Spakula:2014aa}, then $d_\omega(\omega,\omega_k)\in \N\cap [kr,(k+1)r)$ for each $k$, and in particular, the subset $S=\{\omega_k\mid k\in \N\}$ of the limit space $X(\omega)$ of \cite[Definition~3.10]{Spakula:2014aa} is a coarsely embedded copy of $\N$.

Now, define an operator $v$ on $\ell^2(X)$ by 
$$
v:\delta_x=\left\{\begin{array}{ll} \delta_{x_{k+1}^{(n)}}, & x=x_k^{(n)} \text{ for some }n \text{ and } k<m_n \\ 0, & x=x_{m_n}^{(n)}\text{ for some n} \\ \delta_x, & \text{otherwise} \end{array}\right.~;
$$
in words $v$ shifts right by `one unit' along each $S_n$, and acts as the identity elsewhere.  Clearly then $v$ is in $C^*_u(X)$.  Let $\Phi_\omega:C^*_u(X)\to C^*_u(X(\omega))$ be the $*$-homomorphism of \cite[Theorem~4.10]{Spakula:2014aa} (see also \cite[Remark~4.11]{Spakula:2014aa}).  Computing using \cite[Definition~4.1]{Spakula:2014aa}, one sees that $\Phi_\omega(v)$ is the operator $1_{X(\omega)\setminus S}+u$, where $u$ is the unilateral shift along the copy $S$ of $\N$ inside $X(\omega)$, and $1_{X(\omega)\setminus S}$ is the characteristic function of $X(\omega)\setminus S$.  In particular, the image $\Phi_\omega(C^*_u(X))$ contains the proper isometry $\Phi_\omega(v)$ and we are done.
\end{proof}

\begin{proof}[Proof of Theorem \ref{ad0the}, (9) implies (1)]
We will prove the contrapositive, so assume $X$ does not have asymptotic dimension zero.  Hence there exists $r>0$ and a sequence $(S_n)$ of subsets of $X$ with the properties in Lemma \ref{asdim0}.  With notation as in that lemma, define
$$
A:=\bigcup_{n=1}^\infty \{x_1^{(n)},....,x_{m_n-1}^{(n)}\}, \quad  B:=\bigcup_{n=1}^\infty \{x_2^{(n)},....,x_{m_n}^{(n)}\}, \quad \quad \text{and}\quad C:=\bigcup_{n=1}^\infty \{x_1^{(n)},....,x_{m_n}^{(n)}\}
$$
Let $p$ be the characteristic function of $A\cup(X\setminus C)$ and $q$ be the characteristic function of $B\cup (X\setminus C)$.  Then $p$ and $q$ are Murray-von-Neumann equivalent: indeed, the partial isometry defined by 
$$
v:\delta_x \mapsto \left\{\begin{array}{ll} \delta_x & x\in X\setminus C \\ \delta_{x^{(n)}_{i+1}} & x=x^{(n)}_i \text{ for some $n$ and $i\in \{1,...,m_n-1\}$} \\ 0 & x=x^{(n)}_{m_n} \text{ for some }n\end{array}\right.
$$
is in $C^*_u(X)$ by the condition in part (1) of Lemma \ref{asdim0}, and implements such an equivalence.  Assume for contradiction that $C^*_u(X)$ has cancellation.  It follows that $1-p$ and $1-q$ are also Murray-von-Neumann equivalent, say implemented by some partial isometry $w\in C^*_u(X)$ with $ww^*=1-p$ and $w^*w=1-q$.  Note that $1-q$ and $1-p$ are respectively the characteristic functions of the (infinite) subsets 
$$
D=\{x_1^{(n)}\mid n\in \N\}\quad  \text{and}\quad E=\{x_{m_n}^{(n)}\mid n\in \N\}
$$
of $X$.   Now, as $w\in C^*_u(X)$ there exists $s>0$ and $a\in C^*_u(X)$ with propagation $s>0$ and $\|w-a\|<1/2$.  Let $x_1^{(n)}$ be such that $d(x_1^{(n)},E)>s$ (which exists by condition (2) in Lemma \ref{asdim0}).  Let $e\in C^*_u(X)$ be the rank one projection corresponding to the characteristic function of $x_1^{(n)}$, and note that $(1-p)we$ has norm one, but $(1-p)ae=0$.  This contradicts that $\|w-a\|<1/2$, completing the proof.
\end{proof}

\begin{rem}\label{srrem}
Theorem \ref{ad0the} gives a characterization of when $C^*_u(X)$  has stable rank one in terms of the geometry of $X$.  It is natural to ask what other stable ranks can occur for uniform Roe algebras, and to geometrically characterize this.  We cannot say much here, but classical results of Rieffel at least let us give some more examples.  

First note that if $X$ is a non-amenable bounded geometry space, then $C^*_u(X)$ is properly infinite (by Theorem \ref{p.i-para}, for example), and thus has infinite stable rank by \cite[Proposition 6.5]{Rieffel:1983aa}.  On the other hand, the stable rank of $C^*_u|\Z^n|$ equals $2$ for all $n$.  To see this, note first that iterating \cite[Theorem 7.1]{Rieffel:1983aa} on crossed products by $\Z$ proves that $C^*_u|\Z^n|$ has finite stable rank.  On the other hand, considering the corner of $C^*_u|\Z^n|$ corresponding to the characteristic function of the subspace $\{(x_1,...,x_n)\in \Z^n\mid x_1\in k\Z\}$ of $\Z^n$, it is not difficult to see that $C^*_u|\Z^n|\cong M_k(C^*_u|\Z^n|)$ for any $k\geq 1$ and $n\geq 1$; applying \cite[Theorem 6.1]{Rieffel:1983aa} on the behaviour of stable rank under taking matrix algebras forces the stable rank of $C^*_u|\Z^n|$ to be at most two; on the other hand, Theorem \ref{ad0the} implies that it is at least two.

We do not know any examples where the stable rank of $C^*_u(X)$ takes values other than $1$, $2$, or $\infty$: it seems plausible that one can always use some relation of the `matrix isomorphism' above to show that these are the only possible values, but we did not seriously try to pursue this.
\end{rem}

\begin{rem}\label{nucdim rem}
In \cite[Section 8]{Winter:2010eb}, Winter and Zacharias prove that the nuclear dimension of $C^*_u(X)$ is always no more than the asymptotic dimension of $X$.  Theorem \ref{ad0the} above implies in particular that one actually has equality for spaces $X$ of asymptotic dimension zero or one.  Whether or not equality holds more generally than this, or even whether one can get non-trivial lower bounds on the nuclear dimension of $C^*_u(X)$ in some special cases, seems a very interesting question.  As far as we know, this is currently completely open.  

Note also that finite decomposition rank and finite nuclear dimension are quite different properties for uniform Roe algebras by these results and Theorem \ref{ad0the}.   In particular, combining the main result of \cite[Section 8]{Winter:2010eb}, \cite[Theorem 2.3]{Wei:2011kl}, and (for example) the main result of \cite{Delabie:2017vf} gives many examples of quasidiagonal uniform Roe algebras with infinite decomposition rank and finite nuclear dimension (these examples are not, however, strongly quasidiagonal). 
\end{rem}

\section{Real rank zero}\label{rr0 sec}

In this section, we address the question of when $C^*_u(X)$ has real rank zero; see \cite{Brown:1991gf} for definitions and basic results around real rank zero.  Unlike in the previous section, we have no definitive answers here, but the following result, and the corollaries we obtain below, seem to suggest that real rank zero for uniform Roe algebras is quite rare.

\begin{thm}\label{rr0z2}
The uniform Roe algebra of $\Z^2$ does not have real rank zero.  
\end{thm}

We will prove this through some $K$-theoretic computations.  The main idea -- the second author thanks Chris Phillips for suggesting that we exploit this -- is to show that there is an ideal $I$ in $C^*_u|\Z^2|$ and projections in $2$-by-$2$ matrices over the quotient that do not lift to projections in $M_2(C^*_u|\Z^2|)$.  The results of \cite[Theorem 3.14 and Theorem 2.10]{Brown:1991gf} then imply that the real rank of $C^*_u|\Z^2|$ is not zero.  The ideal $I$ we will use is the one generated by the characteristic function of the `$x$-axis' $X:=\{(x,y)\in \Z^2\mid y=0\}$.  Most of the rest of this section will be spent proving the following lemma.

\begin{lem}\label{no lift}
With notation as above, there are projections in $M_2(C^*_u|\Z^2|/I)$ that do not lift to projections in $M_2(C^*_u|\Z^2|)$.
\end{lem}

As it simplifies the $K$-theoretic computations involved (it is not strictly necessary), we will actually work mainly in the Roe algebra $C^*|\Z^2|$, not $C^*_u|\Z^2|$, which is defined as follows (compare for example \cite[Chapter 6]{Higson:2000bs}).  

\begin{defi}\label{roe alg}
Let $X$ be a proper (i.e.\ closed balls are compact) metric space.  Let $H$ be a separable Hilbert space equipped with a non-degenerate representation of $C_0(X)$, which is also \emph{ample}, meaning that no non-zero element of $C_0(X)$ acts as a compact operator.  An operator $a\in \mathcal{B}(H)$ has \emph{finite propagation} if there is $r>0$ such that $fag=0$ whenever $f,g\in C_0(X)$ satisfy $d(\text{supp}(f),\text{supp}(g))>r$, and is \emph{locally compact} if $fa$ and $af$ are compact operators for all $f\in C_0(X)$.

The \emph{Roe algebra} $C^*(X)$ of $X$ is the $C^*$-subalgebra of $\mathcal{B}(H)$ generated by the finite propagation, locally compact operators.  
\end{defi}

If $X=G$ is a group equipped with some choice of left-invariant bounded geometry metric, then we write $C^*|G|$ for its Roe algebra to avoid confusion with the group $C^*$-algebra $C^*(G)$.  In this case, $C^*|G|$ is isomorphic to $\ell^\infty(G,\mathcal{K})\rtimes_r G$, where $\mathcal{K}$ denotes the compact operators on a separable infinite dimensional Hilbert space, and $G$ acts on $\ell^\infty(G,\mathcal{K})$ by the action induced by the right translation action of $G$ on itself.

The Roe algebra is independent of the choice of $H$ up to $*$-isomorphism, so it is customary to ignore $H$ in the notation (even though the $*$-isomorphisms one gets are non-canonical, only becoming so on the level of $K$-theory \cite[Corollary 6.3.13]{Higson:2000bs}).  For concreteness, we consider the Roe algebra of $\Z^2$ to be represented on $\ell^2(\Z^2)\otimes \ell^2(\N)$, noting that the natural multiplication representation of $C_0(\Z^2)$ is nondegenerate and ample.  The isometric inclusion 
$$
\ell^2(\Z^2)\to \ell^2(\Z^2)\otimes \ell^2(\N),\quad v\mapsto v\otimes \delta_{0}
$$
then induces an inclusion $C^*_u|\Z^2|\to C^*|\Z^2|$ as a corner.  Note that with $X=\{(x,y)\in \Z^2\mid y=0\}$ as above, the characteristic function $1_X$ of $X$ acts by multiplication on $\ell^2(\Z^2)\otimes \ell^2(\N)$, and that this makes it an element of the multiplier algebra of $C^*|\Z^2|$.  It thus makes sense to speak of the ideal $J$ in $C^*|\Z^2|$ generated by $1_X$: precisely, this means the closed ideal generated by $\{a1_Xb\mid a,b\in C^*|\Z^2|\}$.   

Before proving Lemma \ref{no lift} we need a preliminary $K$-theoretic lemma.  To state it, note that if $\mathbb{T}^2$ is the $2$-torus, then having $C(\mathbb{T}^2)\cong C^*(\Z^2)$ act on $\ell^2(\Z^2)$ via the regular representation of $\Z^2$ identifies $C(\mathbb{T}^2)$ with a sub-$C^*$-algebra of $C^*_u|\Z^2|$ in a canonical way.  The next lemma is no doubt well-known, but we include a proof as we could not find a good reference in the literature.  We have tried to keep the proof within `standard $C^*$-algebra $K$-theory' at the expense of making it a little long-winded; shorter proofs are certainly possible, using for example the Riemann-Roch theorem, or pairings with cyclic cohomology.

\begin{lem}\label{bott lem}
We have $K_0(C^*|\Z^2|)\cong K_1(J)\cong \Z$ and $K_1(C^*|\Z^2|)=K_0(J)=0$.  Moreover, under the isomorphism $\Z\cong K_0(C^*|\Z^2|)$, the element $n$ can be represented by any element of $K_0(C(\mathbb{T}^2))$ with first Chern class equal to $n$.
\end{lem}

The lemma generalizes to higher dimensional groups $\Z^n$ in a natural way: we give it only for $n=2$ for the sake of simplicity in both the proof and statement.

\begin{proof}
We first note that $C^*|\Z|\cong 1_XC^*|\Z^2|1_X$ is a full corner in $J$ by definition, whence $K_i(J)\cong K_i(C^*|\Z|)$ using the results of \cite[Section 5]{Exel:1993pt}.  On the other hand, $K_i(C^*|\Z|)$ is computed in \cite[Proposition 4.9]{Roe:1996dn} using a Pimsner-Voiculescu argument, which gives the computation of $K_i(J)$ in the statement. 

To compute $K_i(C^*|\Z^2|)$ and identify a generator of $K_0$, we use a Pimsner-Voiculescu argument again: while there are certainly other ways to compute the $K$-theory groups abstractly (see for example \cite[Proposition 6.4.10]{Higson:2000bs}), this seems the easiest way to simultaneously identify a generator of $K_0$.  Let then $A=\ell^\infty(\Z^2,\mathcal{K})$, where $\mathcal{K}$ is a copy of the compact operators on a separable infinite dimensional Hilbert space, so that $C^*|\Z^2|$ is isomorphic to $(A\rtimes\Z)\rtimes \Z$, with the first and second copies of $\Z$ acting by shifting in the first and second coordinates respectively; call these actions $\alpha$ and $\beta$ respectively.  Note that $K_0(A)$ identifies with the group $\Z^{\Z^2}$ of all functions from $\Z^2$ to $\Z$: indeed, there is a canonical isomorphism determined by the fact that it takes the class of a projection $p\in M_n(\ell^\infty(\Z^2,\mathcal{K}))\cong \ell^\infty(\Z^2,\mathcal{K})$ to the function $f:\Z^2\to \Z$ defined by $f(m)=\text{rank}(p(m))$.  Moreover, one has that $K_1(A)=0$, for essentially the same reason that $K_1(\mathcal{K})=0$.  Applying the first Pimsner-Voiculescu sequence gives
$$
\xymatrix{ K_0(A) \ar[r]^-{1-\alpha_*} & K_0(A) \ar[r] & K_0(A\rtimes \Z) \ar[d] \\ K_1(A\rtimes \Z)\ar[u] & K_1(A) \ar[l] & K_1(A) \ar[l]^-{1-\alpha_*}}.
$$
Elementary computations show that the top left copy of $1-\alpha_*$ is surjective with kernel given by the collection of functions from $\Z^2$ to $\Z$ that are constant in the first coordinate direction; thus $K_1(A\rtimes \Z)$ is isomorphic to the group $\Z^\Z$ of all functions from $f:\Z\to\Z$.  Given such a function $f:\Z\to \Z$, define a unitary operator $u_f$ on $\ell^2(\Z)\otimes \ell^2(\N)$ by the formula
\begin{equation}\label{usubf}
u_f:\delta_{(x,y)}\otimes \delta_n\mapsto \left\{\begin{array}{ll} \delta_{(x+1,y)}\otimes \delta_n & f(x)>0 \text{ and } 1\leq n\leq f(x) \\ \delta_{(x-1,y)}\otimes \delta_n & f(n)<0 \text{ and } 1\leq n \leq |f(x)| \\ \delta_{(x,y)}\otimes \delta_n & \text{otherwise} \end{array}\right.,
\end{equation}
which makes $u_f$ an element of the unitization of $A\rtimes \Z$.  A direct computation using the explicit description of the Pimsner-Voiculescu sequence as arising from a Toeplitz-like sequence (see for example \cite[Section 5.2]{Cuntz:2007sf}) and an explicit description of the index map in $K$-theory (see for example \cite[Section 9.2]{Rordam:2000mz}) one computes that the class in $K_1(A\rtimes \Z)$ corresponding to a function $f:\Z\to \Z$ under this isomorphism can be represented by $u_f$.

Now, we consider the Pimsner-Voiculescu sequence associated to the second crossed product.  This gives 
$$
\xymatrix{ K_0(A\rtimes \Z) \ar[r]^-{1-\beta_*} & K_0(A\rtimes \Z) \ar[r] & K_0(C^*|\Z^2|) \ar[d] \\ K_1(C^*|\Z^2|)\ar[u] & K_1(A\rtimes \Z) \ar[l] & K_1(A\rtimes \Z) \ar[l]^-{1-\beta_*} ~,}
$$
and using our computations so far, this simplifies to 
\begin{equation}\label{pv2}
\xymatrix{ 0 \ar[r]^-{1-\beta_*} & 0 \ar[r] & K_0(C^*|\Z^2|) \ar[d] \\ K_1(C^*|\Z^2|)\ar[u] & \Z^\Z \ar[l] & \Z^\Z \ar[l]^-{1-\beta_*} ~,}
\end{equation}
with $\beta$ acting via the shift on the domains of functions in $\Z^\Z$.  Elementary computations show that the bottom right $1-\beta_*$ is surjective, with kernel the constant functions from $\Z$ to $\Z$, completing the computation of $K_i(C^*|\Z^2|)$ as abstract abelian groups.  

It remains to compute a generator of $K_0(C^*|\Z^2|)$.  For this, equip $\mathcal{K}$ with the trivial $\Z^2$ action, and consider the inclusion $\mathcal{K}\to A=\ell^\infty(\Z^2,\mathcal{K})$ as constant functions.  Note that $\mathcal{K}\rtimes \Z\cong C(\mathbb{T},\mathcal{K})$ and that $\mathcal{K}\rtimes \Z^2\cong C(\mathbb{T}^2,\mathcal{K})$ for these actions.  The Pimsner-Voiculescu sequence associated to the second crossed product in this case looks like 
$$
\xymatrix{ K_0(C(\T)) \ar[r]^-{1-\beta_*} & K_0(C(\T)) \ar[r] & K_0(C(\T^2)) \ar[d] \\ K_1(C(\T^2))\ar[u] & K_1(C(\T)) \ar[l] & K_1(C(\T)) \ar[l]^-{1-\beta_*} ~;}
$$
as $\beta$ is the identity map in this case, the arrows labeled $1-\beta_*$ are both zero.  Hence combining the right hand portion of this diagram with that of line \eqref{pv2} above and using naturality of the Pimsner-Voiculescu sequence gives a commutative diagram 
$$
\xymatrix{ 0 \ar[r] & K_0(C(\T)) \ar[d] \ar[r] & K_0(C(\T^2)) \ar[d] \ar[r] & K_1(C(\T)) \ar[d] \ar[r] & 0 & \\
& 0 \ar[r] & K_0(C^*|\Z^2|) \ar[r] &  \ar[r]^-{1-\beta_*} \Z^\Z &  \Z^\Z  \ar[r] & 0 ~.}
$$
Adopting notation as in line \eqref{usubf} above, let $1:\Z\to \Z$ is the constant function with value $1$ everywhere and $u_1$ the associated unitary in the unitization of $C^*|\Z^2|$.  Then the group $K_1(C(\T))\cong \Z$ is generated by the class of a unitary that maps to $[u_1]\in K_1(A\rtimes \Z)$.  Note that $[u_1]$ generates of the kernel of $1-\beta_*$.  On the other hand, $K_0(C(\T))\cong \Z$ is generated by the class $[1]$ of the identity, while $K_0(C(\T^2))\cong \Z\oplus \Z$ is generated by $[1]$ and any element of Chern class one.  The result follows from these comments and a diagram chase.
\end{proof}

\begin{proof}[Proof of Lemma \ref{no lift}]
Let $U=\{(x,y)\in \Z^2\mid y>0\}$ and $L=\{(x,y)\in \Z^2\mid y\leq 0\}$.  Then the images of the characteristic functions $1_U$ and $1_L$ in $C^*_u|\Z^2|/I$ are central: to see this, note that $C^*_u|\Z^2|$ is generated by $\ell^\infty(\Z^2)$ and the bilateral shifts $u$ and $v$ in the first and second coordinate directions respectively; $1_U$ and $1_L$ commute with $\ell^\infty(\Z)$ and $u, v$ in $C^*_u|\Z^2|$, while it is not difficult to check that the commutators $[1_U,v]$ and $[1_L,v]$ are in $I$.  

Now, let $b\in M_2(C(\T^2))$ be a projection associated to any line bundle on $\T^2$ with Chern class one; we identity $b$ with its image in $M_2(C^*|\Z^2|)$; moreover by Lemma \ref{bott lem}, the class of $b$ is a generator of $K_0(C^*|\Z^2|)$.   It follows that the images $p$, $q$ respectively of $(1_U\otimes 1_{M_2})b$ and $(1_L\otimes 1_{M_2})b$ are projections in $M_2(C^*_u|\Z^2|/I)$.  Assume for contradiction that they lift to projections $\widetilde{p}$ and $\widetilde{q}$ in $M_2(C^*_u|\Z^2|)$.  Abusing notation, we use the same symbols for the images of these projections in $M_2(C^*|\Z^2|)$ and $M_2(C^*|\Z^2|/J)$.

Now, using Lemma \ref{bott lem}, the six-term exact sequence in $K$-theory 
$$
\xymatrix{ K_0(J) \ar[r] & K_0(C^*|\Z^2|) \ar[r]^-{\pi_*} & K_0(C^*|\Z^2|/J) \ar[d] \\ K_1(C^*|\Z^2|/J) \ar[u] & K_1(C^*|\Z^2|) \ar[l] & K_1(J) \ar[l] }
$$  
simplifies to 
$$
\xymatrix{ 0 \ar[r] & \Z \ar[r]^-{\pi_*} & \Z^2 \ar[d]   \\ 0 \ar[u] & 0 \ar[l] & \Z\ar[l] }
$$
with the copy of $\Z$ in the middle of the top row generated by $[b]$.  As $\pi_*[b]=[p]+[q]$ and $\pi_*$ is injective, we have that $[b]=[\widetilde{p}]+[\widetilde{q}]$.  Let $\alpha:C^*|\Z^2|\to C^*|\Z^2|$ be the $*$-automorphism induced by the the map $(x,y)\mapsto (x,-y)$ on $\Z^2$; this descends to the quotient by $J$, and abusing notation we will write $\alpha$ for the induced map on the quotient, and also on the uniform versions of these objects.  Noting that $\alpha_*[b]$ corresponds to a line bundle with first Chern class equal to $-1$, and using Lemma \ref{bott lem} we see that $\alpha_*[b]=-[b]$.  It follows from this, the fact that $\alpha(1_U)=1_L$ modulo $I$, and the fact that $1_U$ and $1_L$ are central in $C^*_u|\Z^2|/I$ that $\alpha_*[p]=-[q]$.  Hence by injectivity of $\pi_*$ again, we have $\alpha_*[\widetilde{p}]=-[\widetilde{q}]$.  On the other hand, as $K_*(C^*|\Z^2|)\cong \Z$ is generated by $[b]$, there must exist $n\in \Z$ with $[\widetilde{p}]=n[b]$.  Putting this discussion together
$$
-n[b]=\alpha_*(n[b])=\alpha_*[\widetilde{p}]=-[\widetilde{q}]=-([b]-[\widetilde{p}])=(n-1)[b].
$$
This forces $1-n=n$, which is impossible, so we have our contradiction.
\end{proof}

We conclude this section with a few comments and corollaries.  

\begin{cor}\label{cemb cor}
Say $X$ is a bounded geometry metric space such that there is a coarse embedding $f:\Z^2\to X$.  Then $C^*_u(X)$ does not have real rank zero.  
\end{cor}

\begin{proof}
It is straightforward to check that there is a subspace $Y\subseteq \Z^2$ and $c>0$ such that every $n\in \Z^2$ is within distance $c$ of some $y\in Y$, and such that $f$ restricts to an injection on $Y$.  Note that $C^*_u(Y)$ is Morita equivalent to $C^*_u|\Z^2|$ by \cite[Theorem 4]{Brodzki:2007mi}, and thus $C^*_u(Y)$ does not have real rank zero by  \cite[Theorem 3.8]{Brown:1991gf} and Theorem \ref{rr0z2}.  On the other hand, if $Z:=f(Y)\subseteq X$ then $f$ induces an isomorphism $C^*_u(Y)\cong C^*_u(Z)$, and thus $C^*_u(Z)$ does not have real rank zero.  As $C^*_u(Z)$ identifies with the corner $1_ZC^*_u(X)1_Z$ and real rank zero passes to corners by \cite[Theorem 2.5]{Brown:1991gf}, we are done.
\end{proof}

In particular, note that any subgroup inclusion $\Z^2\to G$ of $\Z^2$ into a countable group is a coarse embedding, so Corollary \ref{cemb cor} implies that the uniform Roe algebra $C^*_u|G|$ of any such group $G$ cannot have real rank zero: examples include $\Z^n$ of course, but many other interesting groups such as $SL(n,\Z)$ for any $n\geq 3$ and the Thompson groups $F$, $T$, and $V$.  On the other hand, there are many groups that admit a coarse embedding of $\Z^2$ that do not have $\Z^2$ as a subgroup.  We give two illustrative examples; much more could be said here.

\begin{exem}\label{rr0ex}
Say $G$ is the fundamental group of a closed hyperbolic $3$-manifold $M$.  Then $G$ does not contain $\Z^2$ as a subgroup as (for example) it is word hyperbolic.  Nonetheless, $G$ is coarsely equivalent to the universal cover of $M$, which is hyperbolic $3$-space $\mathbb{H}^3$, and there are natural coarse embeddings $\R^2\to \mathbb{H}^3$ onto horospheres; these give rise to coarse embeddings $\Z^2\to G$ and thus $C^*_u|G|$ does not have real rank zero.  
\end{exem}

\begin{exem}\label{rr0ex2}
Say $G$ is the (first) Grigorchuk group, which is a torsion group so does not contain $\Z^2$ as a subgroup.  Nonetheless, the arguments of \cite{Smith:2007uq} give rise to coarse embeddings $\Z^2\to G$, and thus $C^*_u|G|$ does not have real rank zero.  
\end{exem}

Theorem \ref{rr0z2} can also be adapted to other examples that are rather far from being groups. 

\begin{exem}\label{rr0ex3}
Fix $n\geq 2$ and for each $k\in \N$, let $X_k=\{m\in \Z^n\mid |m|\leq k\}$ equipped with the subspace metric.  Let $X:=\bigsqcup_{k\in \N}X_k$, equipped with any metric that restricts to the given metric on $X_k$, and that satisfies $d(X_k,X\setminus X_k)\to\infty$ as $k\to\infty$.  Then using the limit space machinery of \cite{Spakula:2014aa}, one can show that $C^*_u(X)$ admits a quotient $*$-homomorphism onto $C^*_u|\Z^n|$.  As real rank zero clearly passes to quotients, this forces the real rank of $C^*_u(X)$ to be positive.  Similar arguments apply to show that if $X$ is any box space of $\Z^n$ (or of any group containing a coarsely embedded copy of $\Z^n$) then $C^*_u(X)$ has positive real rank.
\end{exem}

The technique used to prove Theorem \ref{rr0z2} can be adapted to show that if $G$ is the fundamental group of any closed surface, then $C^*_u|G|$ does not have real rank zero.  We do not know any groups of asymptotic dimension at least two that do not contain coarsely embedded copies of the fundamental group of some closed surface (including $\Z^2=\pi_1(\T^2)$).  Indeed, there are some interesting results in this direction: for example, \cite[Theorem 1]{Bonk:2005rt} implies that a word hyperbolic group has asymptotic dimension at least two if and only if it contains a  coarsely embedded copy of the fundamental group of the closed oriented genus two surface (thanks to Erik Guentner for pointing out this reference).

The following conjecture is thus natural; note that the above discussion implies that it is true if $X$ is restricted to the class of word hyperbolic groups.

\begin{conjecture2}
If $X$ is a bounded geometry metric space of asymptotic dimension at least two, then $C^*_u(X)$ does not have real rank zero. 
\end{conjecture2}

On the other hand, if $X$ has asymptotic dimension zero, then Theorem \ref{ad0the} implies that $C^*_u(X)$ does have real rank zero.  However, we do not know what happens for any $X$ of asymptotic dimension one, such as $\Z$ or a non-abelian free group.  For example, it seems interesting to ask what the real rank of the properly infinite (see \cite[Proposition~5.5]{MR2873171} or Theorem~\ref{p.i-para}) $C^*$-algebra $C^*_u|F_2|$ is; in relation to the discussion above, note that $F_2$ is the unique word hyperbolic group of asymptotic dimension one, up to bijective coarse equivalence.  The following question also seems particularly natural.

\begin{question}\label{rr0zq}
Does $C^*_u|\Z|$ have real rank zero? 
\end{question}

Note that $C^*_u|\Z|$ has real rank zero if and only if $C^*_u(\N)$, which is the classical $C^*$-algebra of weighted shift operators, has real rank zero.  Either a positive or a negative answer to Question \ref{rr0zq} would be interesting.  Indeed, if the answer is `no', then the ideas in the proof of Theorem \ref{ad0the} combined with the limit space machinery of \cite{Spakula:2014aa} would show that if $X$ has positive asymptotic dimension, then $C^*_u(X)$ has positive real rank; combined with Theorem \ref{ad0the}, this would completely settle the question of which uniform Roe algebras have real rank zero: indeed, real rank zero would then join the properties in Theorem 2.2 that are equivalent to asymptotic dimension zero.  On the other hand, if the answer is `yes' then let $X$ be as in Example \ref{rr0ex3}, except starting the construction with $\Z$ rather than $\Z^n$ for $n\geq 2$.  Then $C^*_u(X)$ is a corner in $C^*_u|\Z|$ and would thus also have real rank zero.  The results of \cite{Wei:2011kl} imply that $C^*_u(X)$ is stably finite; $C^*_u(X)$ would thus be an example of a stably finite, real rank zero $C^*$-algebra which by Theorem \ref{ad0the} does not have cancellation.  This would answer a well-known open question: see for example \cite[page 455]{MR2188261}.  

In connection to Question \ref{rr0zq} above, note that the real rank of $C^*_u|\Z|$ is at least finite: indeed, this is even true for $\Z^n$ as the stable rank of $C_u^*|\Z^n|$ equals 2 for all $n\in \N$ and thus by \cite[Proposition V.3.2.4]{MR2188261} the real rank of $C_u^*|\Z^n|$ is at most 3 for all $n\in \N$.

As a final remark, one can also consider other existence of projection properties for uniform Roe algebras: one interesting such property is the \emph{ideal property} (see for example \cite[Definition 1.5.2]{Rordam:2002cs}), meaning that every ideal in the given $C^*$-algebra is generated by the projections it contains; this is much weaker than having real rank zero in general.  Now, uniform Roe algebras have the ideal property when they are nuclear (equivalently, when the underlying space has property A as shown in \cite[Theorem 5.5.7]{Brown:2008qy}), as one can see by applying \cite[Theorem 6.4 and Remark 6.5]{Chen:2004bd}: indeed, these results even imply that each ideal has an approximate unit consisting of diagonal projections.  However, \cite[5.2]{Roe:2013rt} shows that in the non-nuclear case the ideal property can fail for uniform Roe algebras: more precisely the so-called \emph{ghost ideal} need not be generated by the projections it contains.

\begin{question}
Does $C^*_u(X)$ fail the ideal property whenever it is not nuclear?
\end{question}

For example, it is conceivable that the ghost ideal is never generated by the projections that it contains in the absence of property A.  One might also ask about specific classes of examples: for example, is the ghost ideal in the uniform Roe algebra of an expander generated by the projections it contains?

\section{Properly infinite projections in uniform Roe algebras}\label{proj sec}

The aim of this section is to prove some results that relate properly infinite projections in uniform Roe algebras to paradoxical decompositions.  This will be useful when we turn to $K$-theoretic questions later in the paper.   Much of this material is already known in the group case (see below for detailed references), and our proofs are fairly similar to the existing ones.  As it follows in a straightforward way from our results, we also give some new characterizations of supramenability for metric spaces in Proposition \ref{supramenable} below; this will not be used in the rest of the paper, however.

First, we recall the definition of paradoxicality.  

\begin{defi}\label{para dec}
Let $X$ be a bounded geometry metric space.  A \emph{partial translation} on $X$ is a bijection $t:A\to B$ between subsets $A,B\subseteq X$ such that $\sup_{x\in A}d(x,t(x))$ is finite.  A subspace $A$ of $X$ is \emph{paradoxical} if there exist a decomposition $A=A_+\sqcup A_-$ and (bijective) partial translations $t_+:A\to A_+$ and $t_-:A\to A_-$.
\end{defi}

Note that by \cite[Theorem~32]{MR1721355} a bounded geometry metric space is paradoxical if and only if it is non-amenable, where we recall that a bounded geometry metric space is amenable if for any $r,\epsilon>0$ there exists a finite subset $F\subseteq X$ such that 
$$
|\{x\in X\mid d(x,F)\leq r\}|\leq (1+\epsilon)|F|.
$$
In particular, it follows that if $A$ and $B$ are coarsely equivalent in the sense of Definition \ref{cemb}, then $A$ is paradoxical if and only if $B$ is: see for example \cite[Corollary~2.2 and Theorem~3.1]{MR1145337} or \cite[Proposition~3.D.33]{MR3561300}.

We are going to extend the following result given in \cite{ALLW17}.
\begin{thm}[{\cite[Theorem~4.9]{ALLW17}}]\label{p.i-para}
Let $X$ be a bounded geometry metric space. Then the following conditions are equivalent:
\begin{enumerate}
\item $X$ is paradoxical;
\item $C_u^*(X)$ is properly infinite;
\item $M_n(C_u^*(X))$ is properly infinite for some $n\in \N$;
\item $[1]_0=[0]_0$ in $K_0(C_u^*(X))$.
\end{enumerate}
\end{thm}
Recall that a projection $p$ in a $C^*$-algebra $A$ is \emph{properly infinite} if and only if there exist partial isometries $x$ and $y$ in $A$ such that $x^*x=y^*y=p$ and $xx^*+yy^*\leq p$. Moreover, a unital $C^*$-algebra is properly infinite if and only if the unit is properly infinite.

The following proposition extends part of Theorem \ref{p.i-para} above, as well as extending \cite[Proposition~5.5]{MR2873171} from countable discrete groups to metric spaces with bounded geometry.
\begin{prop}\label{Para_subset}
Let $X$ be a bounded geometry metric space and $A$ be a subspace of $X$. Then the following conditions are equivalent:
 \begin{enumerate}
   \item $A$ is a paradoxical subspace of $X$;
   \item $1_A$ is a properly infinite projection in $C_u^*(X)$;
   \item the $n$-fold direct sum $1_A\otimes 1_n$ is properly infinite in $M_n(C^*_u(X))$ for every $n\in \N$;
   \item the $n$-fold direct sum $1_A\otimes 1_n$ is properly infinite in $M_n(C^*_u(X))$ for some $n\in \N$.
  \end{enumerate}
\end{prop}
\begin{proof}
$(1)\Rightarrow (2)$: It follows from Theorem \ref{p.i-para} that the space $A$ is paradoxical if and only if the uniform Roe algebra $C_u^*(A)$ of A is properly infinite. Since the projection $1_A$ is the unit in $C_u^*(A)$, $1_A$ is properly infinite in $C_u^*(A)$. As $C_u^*(A)=1_AC_u^*(X)1_A$ is a corner of $C_u^*(X)$, we may conclude that $1_A$ is also properly infinite in $C_u^*(X)$.

$(2) \Rightarrow (3)$: This follows easily from the definition of proper infiniteness.

$(3) \Rightarrow (4)$: Clear.

$(4) \Rightarrow (1)$: Since $1_A\otimes 1_n$ is properly infinite in $M_n(C_u^*(X))$ for some $n\in \N$, it is still properly infinite in the corner $(1_A\otimes 1_n)M_n(C_u^*(X))(1_A\otimes 1_n)=M_n(C_u^*(A))$. In particular, $M_n(C_u^*(A))$ is a properly infinite $C^*$-algebra.

If we use the discrete metric on $\{1,2,\cdots,n\}$ and the sum metric on the product space $A\times \{1,2,\ldots,n\}$, then we get a natural identification $M_n(C_u^*(A))\cong C_u^*(A\times \{1,2,\ldots,n\})$. Hence, it again follows from Theorem \ref{p.i-para} that the product space $A\times \{1,2,\ldots,n\}$ is paradoxical. As the two spaces $A\times \{1,2,\ldots,n\}$ and $A$ are coarsely equivalent, we conclude that the space $A$ is also paradoxical.
\end{proof}

The following corollary is also known for discrete groups: see \cite[Lemma~5.1]{MR3158244}. The same arguments also work for metric spaces, but we include a proof for the reader's convenience.

\begin{cor}\label{generate ideal}
Let $X$ be a bounded geometry metric space and let $p\in {\ell}^\infty(X)$ and $q\in C_u^*(X)$ be two projections that generate the same closed two-sided ideal in $C^*_u(X)$. If $q$ is properly infinite, then $p$ is also properly infinite in $C_u^*(X)$.
\end{cor}
\begin{proof}
With notation as in \cite[Definition II.3.4.3 (ii)]{MR2188261}, \cite[Corollary II.5.2.11]{MR2188261} implies that $p \precsim q\otimes 1_m$ and $q\precsim p\otimes 1_n$ relatively to $C^*_u(X)$ for some $n,m\in \N$. As $q$ is properly infinite, we have $q\otimes 1_k \precsim q$ for all $k\in \N$. Hence,
\begin{align*}
(p\otimes 1_n)\oplus (p\otimes 1_n)\precsim(q\otimes 1_{mn})\oplus (q\otimes 1_{mn}) \precsim q \precsim p\otimes 1_n.
\end{align*}
It follows that $p\otimes 1_n$ is properly infinite in $M_n(C_u^*(X))$. We conclude from Proposition \ref{Para_subset} that $p$ itself is properly infinite in $C^*_u(X)$.
\end{proof}

The remainder of this section consists of some generalizations of results from \cite{MR3158244}.  This material will not be used in the rest of the paper; however, it is straightforward to demonstrate from the material above so we thought it was worth including.

\begin{defi}
A bounded geometry metric space $X$ is \emph{supramenable} if every non-empty subspace of $X$ is non-paradoxical.
\end{defi}
Equivalently, $X$ is supramenable if and only if every non-empty subspace of $X$ is amenable.  As an example, note that a bounded geometry metric space of subexponential growth is supramenable by \cite[Theorem~66]{MR1721355}. We do not know whether the direct product of two finitely generated supramenable groups is supramenable. However, there exist two supramenable bounded geometry metric spaces such that the direct product is not supramenable \cite[Example~74]{MR1721355}.

We observe that supramenability is invariant under coarse equivalence between metric spaces with bounded geometry as this is the case for amenability (see \cite[Corollary~2.2 and Theorem~3.1]{MR1145337} or \cite[Proposition~3.D.33]{MR3561300}). 

Now we strengthen Proposition \ref{Para_subset} by following the proof in \cite[Proposition~5.3]{MR3158244} for discrete groups as follows.

\begin{prop}\label{supramenable}
Let $X$ be a bounded geometry metric space with property A. Then the following two conditions are equivalent:
\begin{enumerate}
   \item $X$ is supramenable;
   \item the uniform Roe algebra $C_u^*(X)$ contains no properly infinite projections.
  \end{enumerate}
  \end{prop}
\begin{rem}
The property A condition in Proposition~\ref{supramenable} is only used in the proof of the implication $(1)\Rightarrow (2)$ and this condition can be removed if we deal with discrete groups as in \cite[Proposition~5.3]{MR3158244}. The reason property A is not needed for groups is that (supr)amenability implies property A for (countable) discrete groups, but this is not the case for general metric spaces with bounded geometry. For instance, a box space of a finitely generated residually finite group $\Gamma$ is always (supr)amenable. However, such a box space has property A if and only if the group $\Gamma$ is amenable as shown in \cite[Proposition 11.39]{Roe:2003rw}.  On the other hand, we have no counterexample to a version of Proposition \ref{supramenable} without the property A assumption.
\end{rem}

In order to prove Proposition \ref{supramenable}, we need the following theorem due to Gert K. Pedersen.

\begin{thm}[{\cite[Section~5.6]{MR548006} or \cite[Theorem II. 5.2.4]{MR2188261}}]\label{pedersen}
Every $C^*$-algebra $A$ contains a smallest dense two-sided ideal $Ped(A)$, called the \emph{Pedersen ideal} of $A$. Moreover, $Ped(A)$ contains all projections in $A$.\qed
\end{thm}
\begin{proof}[Proof of Proposition \ref{supramenable}]
$(1) \Rightarrow (2)$: Suppose that (2) does not hold, so $C^*_u(X)$ contains a properly infinite projection $p$. If $I$ denotes the closed two-sided ideal in $C^*_u(X)$ generated by $p$, then it follows from \cite[Theorem~6.4 and Remark~6.5]{Chen:2004bd} that $I$ is the closed two-sided ideal generated by $\ell^\infty(X)\cap I$.

Let $\mathbb{A}$ be the directed set of all finite subsets of the set of all projections in $\ell^\infty(X)\cap I$. For each $\alpha\in \mathbb{A}$, let $I_\alpha$ be the closed two-sided ideal in $C^*_u(X)$ generated by $\alpha$. Then
\begin{align*}
I=\bar{\bigcup_{\alpha\in \mathbb{A}}I_\alpha}.
\end{align*}
Since $\bigcup_{\alpha\in \mathbb{A}}I_\alpha$ is a dense two-sided ideal in $I$, it contains the Pedersen ideal $Ped(I)$, and hence it contains all projections in $I$ by Theorem \ref{pedersen}. It follows that $p\in I_\alpha$ for some $\alpha\in \mathbb{A}$ and $I=I_\alpha$. Let $q$ be the supremum of all projections belonging to $\alpha$. Since all these projections belong to $\ell^\infty(X)$, it is easy to see that $I_\alpha$ is the closed two-sided ideal generated by $q$. Hence, $p$ and $q$ generate the same closed two-sided ideal $I$ in $C^*_u(X)$, whence $q$ is properly infinite by Corollary \ref{generate ideal}. Since $q$ belongs to $\ell^\infty(X)$, we may complete the proof using Proposition \ref{Para_subset}.

$(2) \Rightarrow (1)$: If $A$ is any subspace of $X$, then the projection $1_A$ is not properly infinite in $C_u^*(X)$ by assumption. It follows from Proposition \ref{Para_subset} that $A$ is non-paradoxical. Hence, $X$ is supramenable.
\end{proof}
We end this section with a description of supramenability in terms of the existence of injective coarse embeddings of free groups. The following proposition generalizes \cite[Proposition~3.4]{MR3158244} from discrete groups to metric spaces.  To state it, recall that a metric space is \emph{uniformly discrete} if there is $\delta>0$ such that $d(x,y)\geq \delta$ for all distinct $x,y\in X$.  Recall moreover and that a map $f:X\to Y$ between metric spaces is \emph{uniformly expansive} if for all $r>0$, $\sup\{d_Y(f(x_1),f(x_2))\mid x_1,x_2\in X,~d(x_1,x_2)\leq r\}$ is finite; note that uniform expansivity is weaker than being a coarse embedding in the sense of Definition \ref{cemb}.

\begin{prop}\label{supra prop}
Let $X$ be a uniformly discrete bounded geometry metric space. Then the following are equivalent:
 \begin{enumerate}
   \item $X$ is not supramenable;
   \item there is a bi-Lipschitz embedding of the free group $F_2$ into $X$;
   \item there is an injective coarse embedding of the free group $F_2$ into $X$;
   \item there is a coarse embedding of the free group $F_2$ into $X$;
   \item there is an injective uniformly expansive map from the free group $F_2$ into $X$.
  \end{enumerate}
\end{prop}
\begin{proof}
$(1)\Rightarrow (2)$: This follows immediately from \cite[Theorem~6.2]{MR1700742}.\\
$(2)\Rightarrow (3)$:  Every bi-Lipschitz embedding is an injective coarse embedding.\\
$(3)\Rightarrow (4)$: Obvious.\\
$(4)\Rightarrow (1)$: A coarse embedding is a coarse equivalence onto its image.  As noted above, supramenability is invariant under coarse equivalences, so this follows from the fact that $F_2$ is not (supr)amenable.\\
$(3)\Rightarrow (5)$: Every coarse embedding is uniformly expansive.\\
$(5)\Rightarrow (1)$: Let $f:F_2\ra X$ be an injective uniformly expansive map. Since $F_2$ is paradoxical, there exists a decomposition $F_2=A_1\sqcup A_2$ and partial translations $t_i: F_2\ra A_i$ for $i=1,2$. It suffices to show that the non-empty subspace $f(F_2)$ of $X$ is paradoxical.

Since $f$ is injective, it is clear that $f(F_2)=f(A_1\sqcup A_2)=f(A_1)\sqcup f(A_2)$. For $i\in \{1,2\}$, we define $\sigma_i:f(F_2)\ra f(A_i)$ by $\sigma_i(f(x)):=f(t_i(x))$ for $x\in F_2$. It is well-defined by the injectivity of $f$. Moreover, $\sigma_1$ and $\sigma_2$ are clearly bijective and there exist $R_1,R_2>0$ such that $\sup_{x\in F_2}d(x,t_i(x))\leq R_i$. Since $f$ is uniformly expansive, there exist $S_1,S_2>0$ such that $d_X(f(x),f(x'))\leq S_i$ for all $x,x'\in F_2$ satisfying $d(x,x')\leq R_i$. Hence, $d_X(f(x),\sigma_i(f(x)))=d_X(f(x),f(t_i(x)))\leq S_i$ for all $x\in F_2$, which implies that $\sigma_1$ and $\sigma_2$ are partial translations.
\end{proof}

\begin{rem}
The assumption of uniform discreteness in the above is not important: if it were dropped, the same result would hold with (2) replaced by `there is an injective quasi-isometry of the free group $F_2$ into $X$' and the remaining conditions left the same.
\end{rem}

\section{Asymptotic dimension one and $K$-theory}\label{asdim1 sec}

In this section, we turn our attention to $K$-theoretic questions.  Uniform Roe algebras obviously contain many projections coming from the diagonal inclusion $\delta: \ell^\infty(X)\to C^*_u(X)$.  It is natural to ask if this inclusion `sees' all of the $K_0$-group: precisely, under what conditions is the induced map on $K_0$-groups
\begin{equation}\label{kdiag}
\delta_*:K_0(\ell^\infty(X))\to K_0(C^*_u(X))
\end{equation}
surjective?  For us, this question was inspired by work of Elliott and Sierakowski \cite[Section 5]{MR3508496}, who show that surjectivity of this sort holds when $X$ is the underlying metric space of a non-abelian finitely generated free group: their argument uses the Pimsner-Voiculescu exact sequence for free group actions, and we wanted to know if there was some geometric structure underpinning this.  

First, we recall the definition of asymptotic dimension one from Definition \ref{asdim} in the introduction.

\begin{defi}\label{asdim1}
A collection $\mathcal{U}$ of subsets of a metric space is \emph{$r$-separated} if for all distinct $A,B$ in $\mathcal{U}$ we have that $d(A,B)>r$.  A bounded geometry metric space has \emph{asymptotic dimension at most one} if for any $r>0$ there exists a decomposition $X=U\bigsqcup V$, where each of $U$, $V$ in turn decomposes into a uniformly bounded collection of $r$-separated subsets.
\end{defi}

The main result of this section is then as follows.

\begin{thm}\label{delta_sur}
Let $X$ be a metric space with bounded geometry and let $\delta:{\ell}^\infty(X)\ra C_u^*(X)$ be the inclusion of the diagonal. If the asymptotic dimension of $X$ is at most one, then the induced map
\begin{align*}
\delta_*: K_0({\ell}^\infty(X))\ra K_0(C_u^*(X))
\end{align*}
is surjective.
\end{thm}

Philosophically, this can be compared to the fact that if $Y$ is a compact topological space with covering dimension at most one, then $K_0(C(Y))$ is generated by projections in $C(Y)$, in other words by `zero-dimensional information'.  There is indeed a sense in which the image of $\delta_*$ is the `zero-dimensional part' of $K_0(C^*_u(X))$; however, making this precise would involve a detour through the uniform coarse Baum-Connes conjecture, which would take us a little far afield.  In the next section, we will show that $\delta_*$ fails to be surjective for many spaces of asymptotic dimension two or higher, so the above is the best possible general result with a hypothesis in terms of the asymptotic dimension on $X$.  This is analogous to the fact that for many compact spaces $Y$ with covering dimension at least $2$, the $K_0$-group $K_0(C(Y))$ is not generated by projections from $C(Y)$. 

The proof of Theorem \ref{delta_sur} uses a controlled $K$-theory argument that we will come back to later; before doing this, we will study $\delta$ in a more elementary way, and deduce a corollary.   The corollary is again motivated by work of Eliott and Sierakowski: they  point out that surjectivity of $\delta_*$ is particularly interesting when $X$ is a non-amenable group, as in this case one can show that $\delta_*=0$, and thus $K_0(C^*_u(X))=0$.  

The next two lemmas show $\delta_*$ as in line \eqref{kdiag} is zero whenever the space $X$ is non-amenable.  This is a generalization of \cite[Lemma~5.4]{MR3508496}, and has a similar proof.

\begin{lem}\label{full triv}
If $X$ is a paradoxical bounded geometry metric space, then every projection in ${\ell}^\infty(X)$ that is full in $C_u^*(X)$ has the trivial $K_0$-class in $K_0(C_u^*(X))$.
\end{lem}

\begin{proof}
Suppose that $X$ is paradoxical and let $p=1_A$ be any projection in ${\ell}^\infty(X)$ that is full in $C^*_u(X)$. Then the unit $1_X$ in $C^*_u(X)$ is properly infinite and the projections $p$ and $1_X$ generate the same closed two-sided ideal. Hence, $p$ is also properly infinite by Corollary \ref{generate ideal}. Since $C^*_u(A)=pC^*_u(X)p$ is a properly infinite $C^*$-algebra, we have that $[1_A]_0=[0]_0$ in $K_0(C^*_u(A))$ by Theorem \ref{p.i-para}. Since the inclusion map from $C^*_u(A)$ to $C^*_u(X)$ maps the unit $1_A$ in $C^*_u(A)$ to the diagonal projection $1_A$ in $C^*_u(X)$, we conclude that $[p]=[1_A]=[0]$ in $K_0(C^*_u(X))$.
\end{proof}

\begin{lem}\label{para lem}
If $X$ is a paradoxical bounded geometry metric space, then the map
$$
\delta_*:K_0(\ell^\infty(X))\to K_0(C^*_u(X))
$$ 
induced by the diagonal inclusion is zero.
\end{lem}

\begin{proof}
By assumption $X$ admits a partition $X=X_1\sqcup X_2$ such that there exist partial translations $t_i:X\ra X_i$ for $i\in \{1,2\}$. Let $v_i$ be the partial isometry associated to the partial translation $t_i$, so $v_i$ satisfies $v_i^*v_i=1_X$ and $v_iv_i^*=1_{X_i}$ for $i\in \{1,2\}$. Since $1_X=v_i^*1_{X_i}v_i$, we conclude that both $1_{X_1}$ and $1_{X_2}$ are full projections in $C_u^*(X)$.

Let $p$ be any projection in $\ell^{\infty}(X)$ of the form $1_A$ for some subspace $A$ of $X$. In order to show $[p]_0=[0]_0$ in $K_0(C_u^*(X))$, we show that $p$ is a sum of two orthogonal projections with a full complement. It is clear that $p$ is an orthogonal sum of $1_{A\cap X_1}$ and $1_{A\cap X_2}$. Since
\begin{align*}
&& 1_{X_1}=1_X-1_{X_2}\leq 1_X-1_{A\cap X_2}&&\text{and} && 1_{X_2}=1_X-1_{X_1}\leq 1_X-1_{A\cap X_1},&&
\end{align*} 
we conclude that both $1_X-1_{A\cap X_1}$ and $1_X-1_{A\cap X_2}$ are full projections in $C_u^*(X)$. It follows from Lemma~\ref{full triv} that $[1_X-1_{A\cap X_1}]$, $[1_X-1_{A\cap X_2}]$ and $[1_X]$ equal the trivial element of $K_0(C_u^*(X))$. It follows that
\begin{align*}
[p]=[1_{A\cap X_1}]+[1_{A\cap X_2}]=[0]+[0]=[0].
\end{align*}

In order to show that the induced map $\delta_*: K_0({\ell}^\infty(X))\ra K_0(C_u^*(X))$ is zero, we note that there is a group isomorphism
\begin{align*}
\text{dim}: K_0(C(\beta X))\ra C(\beta X,\Z)
\end{align*}
which satisfies that $\text{dim}([p])(x)=\text{Tr}(p(x))$, where $x\in \beta X$, $p$ is a projection in $M_n(C(\beta X))$ for some $n\in \N$ and $\text{Tr}$ denotes the standard non-normalized trace on $M_n(\C)$. In particular, every projection in $M_n(\ell^\infty(X))$ is equivalent to a direct sum of projections from $\ell^\infty(X)$. Since every projection in $\ell^\infty(X)$ defines the trivial $K_0$-class in $K_0(C_u^*(X))$, we see that $\delta_*$ is indeed the zero map.
\end{proof}

As an immediate corollary of Theorem \ref{delta_sur} and Lemma \ref{para lem} we get the following result, which generalizes a result of Elliott and Sierakowski on vanishing of $K_0(C^*_u(X))$ when $X$ is the metric space underlying a finitely generated non-abelian free group.

\begin{cor}\label{van cor}
Let $X$ be a non-amenable bounded geometry metric space with asymptotic dimension (at most) one.  Then $K_0(C^*_u(X))=0$. \qed
\end{cor}

Note that every metric space with asymptotic dimension zero is necessarily amenable, so this corollary has nothing to say about such spaces. Moreover, if $X$ is a non-amenable finitely \emph{presented} group of asymptotic dimension one, then $X$ is a virtually free group \cite[Theorem 61]{Bell:2008fk}; thus Corollary \ref{van cor} does not really go further than the work of Elliott and Sierakowski in the case of finitely presented groups.   There is however a large zoo of finitely generated groups giving interesting examples: for example, any generalized lamplighter group $\Z_2\wr F_n$, $F_n$ the free group on $n>1$ generators, as follows from \cite[Theorem 63]{Bell:2008fk}.  There are also many examples of spaces to which the theorem applies.

We now turn to the proof of Theorem \ref{delta_sur}.  

This uses some controlled $K$-theory machinery from  \cite{Guentner:2014bh}.  In order to keep the length of the current paper under control, we refer to \cite{Guentner:2014bh} for basic definitions that we do not use directly; see also \cite{Oyono-Oyono:2011fk} for further background on controlled $K$-theory.  To give a very rough idea, recall first that a $C^*$-algebra $A$ is \emph{filtered} if it is equipped with a collection of self-adjoint subspaces $(A_r)_{r\geq 0}$ (sometimes also assumed closed, but this is not necessary for us) indexed by the positive real numbers such that:
\begin{enumerate}
\item $A_r\subseteq A_s$ for $s\leq r$; 
\item $A_r\cdot A_s\subseteq A_{r+s}$ for all $r,s$; and 
\item $\bigcup_{r\geq 0}A_r$ is dense in $A$.
\end{enumerate}
A motivating example occurs when $A$ is the uniform Roe algebra of some space filtered by the notion of propagation.   If $A$ is a filtered $C^*$-algebra, the idea of the controlled $K$-theory group $K_0^{r, \epsilon}(A)$ is to build the $K$-theory group in the usual sort of way, but only allowing projections and homotopies from matrices over $A_r$; as matrices over $A_r$ may contain very few projections, to make good sense of this one actually allows projections-up-to-$\epsilon$-error.  The group $K_1^{r,\epsilon}$ is defined similarly, but working with unitaries-up-to-$\epsilon$-error.  Following  \cite{Guentner:2014bh}, we usually just fix $\epsilon=1/8$ in what follows and use the words \emph{quasi-projection} or \emph{quasi-unitary} for operators that are projections or unitaries up to $1/8$-error in the appropriate sense.

The following comes from \cite[Definitions 7.2 and 7.3]{Guentner:2014bh}.   

\begin{defi}\label{uni ex}
Let $\mathcal{K}=\mathcal{K}(\ell^2(\N))$, and $A$ be a $C^*$-algebra.  Let $A\otimes \mathcal{K}$ denote the spatial tensor product of $A$ and $\mathcal{K}$; using the canonical orthonormal basis on $\ell^2(\N)$, we identify elements of $A\otimes \mathcal{K}$ with $\N$-by-$\N$ matrices with entries from $A$.  For a subspace $S$ of $A$, let $S\otimes \mathcal{K}$ denote the subspace of $A\otimes \mathcal{K}$ consisting of matrices with all elements in $S$.  Note that a filtration on $A$ induces a filtration on $A\otimes \mathcal{K}$ by defining $(A\otimes \mathcal{K})_r:=A_r\otimes \mathcal{K}$.

Let $(A^\omega)_{\omega\in \Omega}$ be a collection of filtered $C^*$-algebras, and filter each stabilization $A^\omega\otimes\mathcal{K}$ as above.  For each $\omega\in \Omega$, let $I^\omega$ and $J^\omega$ be $C^*$-ideals in $A^\omega$, which we assume are filtered by the subspaces $I^\omega_r:=A^\omega_r\cap I^\omega$ and $J^\omega_r:=A^\omega_r\cap J^\omega$.  Equip the stabilizations $I^\omega\otimes \mathcal{K}$ and $J^\omega\otimes \mathcal{K}$ with the filtrations defined above.

The collection $(I^\omega,J^\omega)_{\omega\in \Omega}$ of pairs of ideals is \emph{uniformly excisive} if for any $r_0,m_0\geq 0$ and $\epsilon>0$, there are $r\geq r_0$, $m\geq 0$, and $\delta>0$ such that:
\begin{enumerate}[(i)]
\item \label{uni ex i} for any $\omega\in \Omega$ and any $a\in (A^\omega\otimes \mathcal{K})_{r_0}$ of norm at most $m_0$, there exist elements $b\in (I^\omega\otimes \mathcal{K})_r$ and $c\in (J^\omega\otimes \mathcal{K})_r$ of norm at most $m$ such that $\|a-(b+c)\|<\epsilon$;
\item \label{uni ex ii} for any $\omega\in \Omega$ and any $a\in I^\omega\otimes \mathcal{K}\cap J^\omega\otimes \mathcal{K}$ such that 
$$
d(a,(I^{\omega}\otimes \mathcal{K})_{r_0})<\delta \quad \text{and} \quad d(a,(J^\omega\otimes\mathcal{K})_{r_0})<\delta
$$ 
there exists $b\in I^\omega_r\otimes \mathcal{K}\cap J^\omega_r\otimes \mathcal{K}$ such that $\|a-b\|<\epsilon$.
\end{enumerate}
\end{defi}

The point of this definition is the existence of the following `Mayer-Vietoris sequence' from \cite[Proposition 7.4]{Guentner:2014bh}.  See \cite[Section 4]{Guentner:2014bh} for conventions on controlled $K$-theory.

\begin{prop}\label{con mv}
Say $(A^\omega)_{\omega\in \Omega}$ is a collection of non-unital filtered $C^*$-algebras and $(I^\omega,J^\omega)_{\omega\in \Omega}$ a uniformly excisive collection of pairs of non-unital ideals.  Then for any $r_0\geq 0$ there are $r_1,r_2\geq r_0$ with the following property.  For each $\omega$, each $i\in \{0,1\}$ and each $x\in K_i^{r_0,1/8}(A^\omega)$ there is an element
$$
\partial_c(x) \in K_{i+1}^{r_1,1/8}(I^\omega\cap J^\omega)
$$
(here `~$i+1$' is to be understood mod $2$) such that if $\partial_c(x)=0$ then there exist 
$$
y\in K_i^{r_2,1/8}(I^\omega) \quad \text{and} \quad z\in K_i^{r_2,1/8}(J^\omega)
$$
such that 
$$
x=y+z \quad \text{in}\quad K_i^{r_2,1/8}(A^\omega)
$$ 
(here we have abused notation, conflating $y$ and $z$ with the images under the natural maps $K_i^{r_2,1/8}(I^\omega)\to  K_i^{r_2,1/8}(A^\omega)$ and $K_i^{r_2,1/8}(J^\omega)\to  K_i^{r_2,1/8}(A^\omega)$; we will frequently abuse notation in this way below).
\end{prop}

Let now $X$ be an arbitrary bounded geometry metric space.  Recall that the \emph{support} of $a\in C^*_u(X)$ is $\text{supp}(a):=\{(x,y)\in X\times X\mid a_{xy}\neq 0\}$ and the \emph{propagation} is $\text{prop}(a):=\sup\{d(x,y)\mid (x,y)\in \text{supp}(a)\}$.  For any subset $U$ of $X$ and $r\geq 0$, let $U^{(r)}$ denote the $r$-neighbourhood of $U$, i.e.\ $U^{(r)}:=\{x\in X\mid d(x,U)\leq r\}$.  We extend these definitions to elements of $C^*_u(X)\otimes \mathcal{K}$ by considering operators in $C^*_u(X)\otimes \mathcal{K}$ as $\N\times \N$ matrices over $C^*_u(X)$, and defining the support of an element $a$ to be the union of the supports of all its matrix entries.  

Let $\Omega$ be the set of all decompositions $X=U\bigsqcup V$, where $U$ and $V$ are subsets of $X$ such that $U$ and $V$ are equipped with decompositions $U=\bigsqcup_{i\in I}U_i$ and $V=\bigsqcup_{j\in J}V_j$ into subsets.   Given such a pair, for each $r\geq 0$, define subspaces of $C^*_u(X)$ by
$$
B(U)_r:=\Big\{a\in C^*_u(X)\otimes \mathcal{K} \mid \text{supp}(a)\subseteq \bigcup_{i\in I} U_i^{(r)}\times U_i^{(r)} \text{ and } \text{prop}(a)\leq r\Big\},
$$
$$
B(V)_r:=\Big\{a\in C^*_u(X)\otimes \mathcal{K} \mid \text{supp}(a)\subseteq \bigcup_{j\in J} V_j^{(r)}\times V_j^{(r)} \text{ and } \text{prop}(a)\leq r\Big\},
$$
and 
$$
B(U\cap V)_r:=\Big\{T\in C^*_u(X)\otimes \mathcal{K} \mid \text{supp}(a)\subseteq \bigcup_{i\in I,j\in J} (U_i^{(r)}\cap V_j^{(r)})\times (U_i^{(r)}\cap V_j^{(r)})\Big\}
$$
(note that there is no propagation control on elements of $B(U\cap V)_r$).  Define further
$$
A^\omega_r:=B(U)_r+B(V)_r+B(U\cap V)_r
$$
and 
$$
I^\omega_r:=B(U)_r+B(U\cap V)_r, \quad J^\omega_r:=B(V)_r+B(U\cap V)_r
$$
Define $A^\omega:=\overline{\bigcup_{r\geq 0}A^\omega_r}$, $I^\omega:=\overline{\bigcup_{r\geq 0}I^\omega_r}$ and $J^\omega:=\overline{\bigcup_{r\geq 0}J^\omega_r}$.   

\begin{lem}\label{id filt}
Equipped with the structures above, each $A^\omega$ is a $C^*$-algebra filtered by the subspaces $A^\omega_r$.  Moreover, $I^\omega$ and $J^\omega$ are ideals in $A^\omega$, filtered by the subspaces $I^\omega_r$ and $J^\omega_r$ respectively.
\end{lem}

\begin{proof}
For any $r,s\geq 0$, it is straightforward to check the following inclusions of subspaces:
$$
B(U)_r\cdot B(U)_s\subseteq B(U)_{r+s}, \quad B(V)_r\cdot B(V)_s\subseteq B(V)_{r+s}
$$
and 
$$
B(U\cap V)_r\cdot B(U\cap V)_s\subseteq B(U\cap V)_{r+s};
$$
direct checks also give
$$
B(U)_r\cdot B(V)_s\subseteq B(U\cap V)_{r+s}
$$
and 
$$
B(U)_r\cdot B(U\cap V)_s\subseteq B(U\cap V)_{r+s}, \quad B(V)_r\cdot B(U\cap V)_s\subseteq B(U\cap V)_{r+s}.
$$
Combining all these inclusions gives the statement of the lemma.
\end{proof}

Note that each $A^\omega$ is just a copy of $C^*_u(X)\otimes \mathcal{K}$ as a $C^*$-algebra, but with a different filtration from the usual one induced by propagation.  

\begin{lem}\label{uni ex lem}
With notation as above, the collection $(I^\omega, J^\omega)_{\omega\in \Omega}$ is uniformly excisive.  
\end{lem}

\begin{proof}
For notational simplicity, let us ignore the tensored-on copies of $\mathcal{K}$; as the reader can easily verify, this makes no difference to the proof.  For a subset $Y$ of $X$, let as usual $1_Y$ denote the characteristic function of $Y$, considered as a diagonal element of $C^*_u(X)$. 

Let first $r_0\geq 0$, and let $a$ be an element of $A^\omega_{r_0}$.    As $X=U\bigsqcup V$ and $U=\bigsqcup_{i\in I}U_i$, $V=\bigsqcup_{j\in J}V_j$, we have that 
$$
a=\sum_{i\in I}1_{U_i}a+\sum_{j\in J}1_{V_j}a
$$
(convergence in the strong operator topology).  It is straightforward to check that the first term is in $I^\omega_{r_0}$ and the second in $J^\omega_{r_0}$ so this implies the condition (i) from Definition \ref{uni ex}.

Now say we are given $\epsilon>0$ and let $\delta=\epsilon/3$.  Assume $a\in I^\omega\cap J^\omega$ satisfies $d(a,I^\omega_{r_0})<\delta$ and $d(a,J^\omega_{r_0})<\delta$, so there are $a_U\in ,I^\omega_{r_0}$ and $a_V\in J^\omega_{r_0}$ with $\|a-a_U\|<\delta$ and $\|a-a_V\|<\delta$.  Let $\chi$ be the characteristic function of $\bigcup_{i\in I}U_i^{(r_0)}$.  Then $a_U\chi=a_U$, and it is not difficult to check that $a_V\chi$ is in $I^\omega_{2r_0}\cap J^\omega_{2r_0}$.  Moreover,
$$
\|a-a_V1_U\|\leq \|a-a_U\|+\|a_U\chi-a_V\chi\|<\epsilon,
$$
completing the proof.
\end{proof}

We are now ready for the proof of Theorem \ref{delta_sur}.

\begin{proof}[Proof of Theorem \ref{delta_sur}]
It suffices to show that if $p$ is a projection in $M_n(C^*_u(X))$ for some $n$, then the class $[p]\in K_0(C^*_u(X))$ is in the image of $\delta_*$.  Let $q$ be a quasi-projection in $M_n(C^*_u(X))$ that has propagation at most $r_0$ for some $r_0\geq 0$, and that approximates $p$ well enough so that the comparison map of \cite[Definition 4.8]{Guentner:2014bh} takes $[q]\in K_0^{r_0,1/8}(C^*_u(X))$ to $[p]$ (here $C^*_u(X)$ is equipped with the usual filtration by propagation).  

Now, for the uniformly excisive family from Lemma \ref{uni ex lem}, let $r_1$ and $r_2$ be the constants associated to this $r_0$ as in Proposition \ref{con mv}.  Using asymptotic dimension at most one, there exists a decomposition $X=U\bigsqcup V$ such that each of $U=\bigsqcup_{i\in I}U_i$ and $V=\bigsqcup_{j\in J}V_j$ are disjoint unions of uniformly bounded $r$-separated sets, where $r>3r_2$.  Consider the element 
$$
\delta_c[q]\in K_1^{r_1,1/8}(I^\omega\cap J^\omega)=K_1^{1/8}(B(U\cap V)_{r_1}).
$$
Note that $B(U\cap V)_{r_1}$ is actually a $C^*$-algebra, and moreover as $2r_1<r$, it is stably isomorphic to a direct product $\prod_{k\in K} M_{n_k}(\C))$ of matrix algebras of uniformly bounded sizes.  Hence, it is an AF algebra (see \cite[Lemma 8.4]{Winter:2010eb} or the proof of Theorem \ref{ad0the}, part (1) implies (2)) and so
$$
K_1^{r_1,1/8}(B(U\cap V)_{r_1})=K_1(\prod_{k\in K} M_{n_k}(\C)))=0.
$$
In particular, $\partial_c[q]=0$.  

It follows from Proposition \ref{con mv} that there are elements $y\in K_0^{r_2,1/8}(I^\omega)$, $z\in K_0^{r_2,1/8}(J^\omega)$ with $[q]=y+z$ in $K_0^{r_2,1/8}(A^\omega)$.  Passing through the canonical maps
$$
K_0^{r_2,1/8}(A^\omega)\to K_0(A^\omega)\cong K_0(C^*_u(X)),
$$
it follows that $[p]=y+z$ in $K_0(C^*_u(X))$, where we have abused notation by conflating $y$, $z$, and their images under the canonical comparison maps (see \cite[Definition 4.8]{Guentner:2014bh})
$$
K_0^{r_2,1/8}(I^\omega)\to K_0(C^*_u(X)) \quad \text{and} \quad K_0^{r_2,1/8}(J^\omega)\to K_0(C^*_u(X)).
$$
On the other hand, by choice of $r_2$ again, there is a factorization of the first of these maps as 
$$
K_0^{r_2,1/8}(I^\omega)\to K_0\Big(\prod_{i\in I}\mathcal{B}(\ell^2U_i^{(r_2)})\Big)\to K_0(C^*_u(X)),
$$
and similarly for $J$ and $B$.  Noting that $\prod_{i\in I}\mathcal{B}(\ell^2U_i^{(r_2)})$ is a direct product of matrix algebras of bounded size, any element in the middle group above is clearly equivalent to something in the image of the map on $K$-theory induced by the diagonal inclusion 
$$
\ell^\infty (U^{(r_2)})\to \prod_{i\in I}\mathcal{B}(\ell^2U_i^{(r_2)}),
$$
and similarly for $V$. The result follows.
\end{proof}

\section{Some higher dimensional behaviour in $K$-theory}\label{high dim sec}

In the previous section, our main result was that if $X$ has asymptotic dimension at most one, then the inclusion of the diagonal into its uniform Roe algebra induces a surjection on $K$-theory.  Our main result in this section gives examples of spaces with higher asymptotic dimension where this surjectivity fails, including some spaces of asymptotic dimension two.  We also negatively answer the following question of Elliott and Sierakowski:
\begin{question}{\cite[Question~5.2]{MR3508496}}\label{esq}
Is $K_0(C_u^*|G|)$ zero for every non-amenable group $G$?
\end{question}

The following theorem, while by no means best possible, gives a reasonably general result.  The statement uses a little Riemannian geometry: we recommend that readers with no background in this just skip to the special cases  Example \ref{exes}.  To state the theorem, we recall that a Riemannian manifold $M$ has \emph{bounded geometry} if the curvature tensor is uniformly bounded, as are all of its derivatives, and if moreover the injectivity radius of $M$ is bounded below.  Moreover $M$ is \emph{uniformly contractible} if for all $r>0$ there is $s\geq r$ such that for any $x\in M$, the inclusion of the $r$-ball around $x$ into the $s$-ball is null homotopic.  Finally, recall that a subset $X$ of $M$ is a \emph{net} if it is countable, and if there exists some $c>0$ such that for all $m\in M$ there is some $x\in X$ with $d(x,m)<c$.

\begin{thm}\label{cbc the}
Let $M$ be a connected, complete, even-dimensional, bounded geometry, uniformly contractible Riemannian manifold, and assume that the coarse Baum-Connes conjecture holds for $M$.  Let $X$ be a net in $M$.  Then the map
$$
\delta_*:K_0(\ell^\infty(X))\to K_0(C^*_u(X))
$$
is not surjective.
\end{thm}

Before starting on the proof of this, we give some examples.

\begin{exem}\label{exes}
Say $N$ is a closed, connected, even-dimensional Riemannian manifold with non-positive sectional curvature.  Let $X=\pi_1(N)$ be the fundamental group of $N$ considered as a metric space with respect to a word metric associated to a finite generating set.  Take $M$ to be the universal cover of $N$, and consider $X$ as a subset of $M$ via any orbit inclusion.  Then $M$ and $X$ satisfy the assumptions of Theorem \ref{cbc the}: see \cite[Section 12.4]{Higson:2000bs} for a nice explanation of the coarse Baum-Connes conjecture in this case.

Particularly nice special cases include the following:
\begin{enumerate}
\item $X=\Z^n$ for some even $n$: one can take $N=\mathbb{T}^n$ the $n$-torus with its canonical flat metric and $M=\R^n$.
\item $X$ is the fundamental group of a closed surface of genus at least two: one can take $N$ to be the surface equipped with some choice of hyperbolic structure, and $M$ to be the hyperbolic plane.
\end{enumerate}
Note that in the second case $X$ is a non-amenable group with asymptotic dimension two: indeed, it has the same asymptotic dimension as the hyperbolic plane as the two are coarsely equivalent, and one can show that this is $2$ using similar arguments to those used in \cite[Section 9.2]{Roe:2003rw} for $\R^2$.  Such groups give negative answers to Question \ref{esq}.  Many other examples of groups for which the answer to Question \ref{esq} is negative can be found similarly.  
\end{exem}

Note that in cases (1) and (2) in Examples \ref{exes} above, one can also give more direct arguments based for example on cyclic homology, or the proof of Lemma \ref{bott lem} adapted to the uniform Roe algebra, that do not go via the coarse Baum-Connes conjecture; however, the other approaches we know would be longer and give more specialized classes of examples.

We now start to look at the proof of Theorem \ref{cbc the}, and fix notation as in the statement from now on.  Recall the definition of the Roe algebra of $X$ from Definition \ref{roe alg} above.  
Fixing any rank one projection $p_0\in \mathcal{K}$ induces an inclusion
\begin{equation}\label{incl}
C^*_u(X)\to C^*(X),\quad a\mapsto a\otimes p_0.
\end{equation}
We note that there is a $*$-isomorphism $C^*(M)\cong C^*(X)$, which is not canonical, but becomes so on the level of $K$-theory: see for example \cite[Chapter 6]{Higson:2000bs}.

We start with a lemma that is probably well-known to experts; it could be proved under more general hypotheses, but we do not need a maximally general result here.  

\begin{lem}\label{linf0}
The canonical inclusion 
$$
\ell^\infty(X,\mathcal{K})\to C^*(X)
$$
induces the zero map on $K$-theory.
\end{lem}

\begin{proof}
We consider the commutative diagram
$$
\xymatrix{ K_*(M) \ar[r] & K_*(C^*(X)) \\
K_*(X) \ar[r] \ar[u] & K_*(\ell^\infty(X,\mathcal{K})) \ar[u] ~.}
$$
Here the two horizontal maps are the assembly maps (see for example \cite[Chapter 5]{Roe:1996dn}) for the space $M$ (noting that $C^*(X)\cong C^*(M))$, and the space $X$ equipped with the `metric' where all distinct points are infinitely far apart.  As $M$ is uniformly contractible, the top horizontal map identifies with the coarse Baum-Connes assembly map (see for example \cite[Chapter 8]{Roe:1996dn}), and is an isomorphism by assumption.  The bottom horizontal map is also the coarse Baum-Connes assembly map for $X$ with the `infinite metric' above; this is easily seen to be an isomorphism.  The vertical maps are induced by the inclusion of $X$ into $M$, and the diagram commutes by naturality of assembly.    

Now, $K$-homology converts topological disjoint unions to direct products of abelian groups; as $X$ is a discrete space, we thus have that
$$
K_*(X)\cong \prod_{x\in X} K_*(\{x\}).
$$
As $M$ is non-compact, each of the maps $K_*(\{x\})\to K_*(M)$ factors through the $K$-homology group of a ray $[0,\infty)$, which is zero, and is thus the zero map.  Hence the left hand vertical map in the commutative diagram is zero, and we are done.
\end{proof}

\begin{lem}\label{dir lem}
The inclusion 
$$
C^*_u(X)\to C^*(X)
$$
from line \eqref{incl} above is non-zero on $K$-theory.
\end{lem}

\begin{proof}
As $M$ is a contractible, it is spin$^\text{c}$.  As it is moreover even dimensional, Poincar\'{e} duality in $K$-homology implies that the $K$-homology groups of $M$ are $K_0(M)\cong \Z$ and $K_1(M)=0$, with $K_0(M)$ generated by the class $[D]$ of the spin$^\text{c}$ Dirac operator $D$ for some choice of spin$^\text{c}$ structure.  As the coarse Baum-Connes conjecture is true by assumption, and as $M$ has bounded geometry and is uniformly contractible, the image $\mu[D]$ of $[D]$ in $K_0(C^*(X))\cong \Z$ is a generator; here we have used that $C^*(X)\cong C^*(M)$.  Now, as discussed in \cite[Section 3]{Spakula:2009tg}, our assumptions imply that $D$ defines also a class $[D_u]$ in uniform $K$-homology $K_0^u(X)$, and moreover, taking the image of this class under the uniform assembly map (see \cite[Section 9]{Spakula:2009tg} or \cite[Chapter 7]{Engel:2014tx}) gives a class $\mu_u[D_u]\in K_0(C^*_u(X))$.  The map in line \eqref{incl} takes $\mu_u[D_u]$ to the generator $\mu[D]$, as is straightforward from the definitions involved.  This completes the proof.
\end{proof}

\begin{proof}[Proof of Theorem \ref{cbc the}]
Consider the commutative diagram
$$
\xymatrix{ K_0(\ell^\infty(X,\mathcal{K})) \ar[r] & K_0(C^*(X)) \\ K_0(\ell^\infty(X))\ar[u]  \ar[r]^{\delta_*} & K_0(C^*_u(X)) \ar[u] ~,}
$$
where the vertical maps are induced by the formula $a\mapsto a\otimes p_0$ as in line \eqref{incl}, and the horizontal maps are induced by the diagonal inclusions.  Lemma \ref{linf0} implies that the top arrow is zero; as by Lemma~\ref{dir lem}, the right hand vertical map is non-zero, $\delta_*$ cannot be surjective.
\end{proof}

\bibliographystyle{plain}
\bibliography{kangbib,Generalbib}
\end{document}